\documentclass[12pt,reqno]{amsart}
\usepackage{amsmath,amsfonts,amsthm,amsopn,amssymb,mathrsfs,enumerate,color}
\usepackage{amsmath}
\usepackage{comment}

\usepackage{cite,marginnote}
\usepackage{mathtools} 

\usepackage{color,graphicx,enumerate}
\usepackage[colorlinks=true,urlcolor=blue,
citecolor=red,linkcolor=blue,linktocpage,pdfpagelabels,
bookmarksnumbered,bookmarksopen]{hyperref}
\usepackage[english]{babel}

\usepackage[left=2.9cm,right=2.9cm,top=2.8cm,bottom=2.8cm]{geometry}

\numberwithin{equation}{section}

\makeindex
\newtheorem{theorem}{Theorem}[section]

\newtheorem{corollary}{Corollary}[section]
\newtheorem{remark}{Remark}[section]

\newtheorem{definition}{Definition}[section]
\newtheorem{proposition}{Proposition}[section]
\newtheorem{lemma}{Lemma}[section]

\newtheorem{iteration lemma}{iteration Lemma}[section]

\newcommand{\bt}{\begin{theorem}}
\newcommand{\et}{\end{theorem}}
\newcommand{\bl}{\begin{lemma}}
\newcommand{\el}{\end{lemma}}
\newcommand{\bd}{\begin{definition}}
\newcommand{\ed}{\end{definition}}
\newcommand{\bc}{\begin{corollary}}
\newcommand{\ec}{\end{corollary}}
\newcommand{\bp}{\begin{proof}}
\newcommand{\ep}{\end{proof}}
\newcommand{\bx}{\begin{example}}
\newcommand{\ex}{\end{example}}
\newcommand{\bi}{\begin{exercise}}
\newcommand{\ei}{\end{exercise}}
\newcommand{\bo}{\begin{proposition}}
\newcommand{\eo}{\end{proposition}}
\newcommand{\br}{\begin{remark}}
\newcommand{\er}{\end{remark}}
\newcommand{\beq}{\begin{equation}}
\newcommand{\eeq}{\end{equation}}
\newcommand{\ba}{\begin{align}}
\newcommand{\ea}{\end{align}}
\newcommand{\bn}{\begin{enumerate}}
\newcommand{\en}{\end{enumerate}}
\newcommand{\bg}{\begin{align*}}
\newcommand{\bcs}{\begin{cases}}
\newcommand{\ecs}{\end{cases}}

\newcommand{\bean}{\begin{eqnarray*}}
\newcommand{\eean}{\end{eqnarray*}}

\def\bd{\mathrm{bd}\,}

\usepackage{amsmath,amssymb,amsthm,mathtools}
\usepackage{hyperref}

\theoremstyle{remark}

\title[Stability for Affine Sobolev Inequality]{Sharp Quantitative Stability for the Affine \(p\)-Sobolev Inequality, Part I: The Case \(2\le p<n\)}

\author[S.~Fan]{Song Fan}
\author[G-D.~Li]{Gui-Dong Li}
\author[J.~J.~Zhang]{Jianjun Zhang}
	
\address[S.~Fan]{\newline\indent School  of  Mathematics  and  Statistics
\newline\indent
Guizhou University
\newline\indent
Guiyang, 550025, Guizhou, PR China}
\email{\href{mailto:doraemonsong77@gmail.com}{doraemonsong77@gmail.com}}

\address[G-D.~Li]{\newline\indent School  of  Mathematics  and  Statistics
\newline\indent
Guizhou University
\newline\indent
Guiyang, 550025, Guizhou, PR China}
\email{\href{mailto:bestdong123@163.com}{bestdong123@163.com}}

\address[J.~J.~Zhang]{\newline\indent College of Mathematics and Statistics
\newline\indent
Chongqing Jiaotong University
\newline\indent
Xuefu, Nan'an, 400074, Chongqing, PR China}
\email{\href{mailto:zhangjianjun09@tsinghua.org.cn}{zhangjianjun09@tsinghua.org.cn}}

\begin{document}

\begin{abstract}

We prove a sharp quantitative stability  result for the affine
\(L^p\)-Sobolev inequality, for \(p\ge2\), introduced by
Lutwak--Yang--Zhang (\emph{J. Differential Geom.}, \textbf{62} (2002), 17--38).
Moreover, the stability exponent is shown to be optimal, and equal to \(p\).

\vskip0.23in

\noindent{\it   {\bf Key  words:} Sharp quantitative stability, Affine \(L^p\) Sobolev inequality.}

\vskip0.1in
\noindent{\it  {\bf 2020 Mathematics Subject Classification:} Primary 46E35; Secondary 26D10, 35A23.}

\end{abstract}

\maketitle

\section{Introduction}

Sharp Sobolev inequalities, together with their quantitative stability
theory, play a fundamental role in geometric analysis and nonlinear partial
differential equations. In the Euclidean setting, the sharp Sobolev inequality
and the classification of its extremals go back to the seminal works of Aubin
and Talenti~\cite{Aubin76,Talenti76}. Let \(1<p<n\), and set
\( p^*=\frac{np}{n-p}\), \(p'=\frac{p}{p-1}\).
The sharp Sobolev inequality asserts that
\begin{equation}
\label{eq:intro-sobolev}
\|\nabla u\|_{L^p(\mathbb R^n)}^p
\ge
S_{n,p}^p
\|u\|_{L^{p^*}(\mathbb R^n)}^p,
\qquad
u\in\dot W^{1,p}(\mathbb R^n).
\end{equation}
Equality in \eqref{eq:intro-sobolev} holds precisely for the Talenti family
\cite{Talenti76,Aubin76}
\begin{equation*}
U_{c,\lambda,x_0}(x)
=
c\,\lambda^{\frac{n-p}{p}}
\left(1+\lambda^{p'}|x-x_0|^{p'}\right)^{-\frac{n-p}{p}},
\qquad
c\in\mathbb R,\quad \lambda>0,\quad x_0\in\mathbb R^n .
\end{equation*}
Equivalently, after fixing the normalized bubble
\[
U(x):=\left(1+|x|^{p'}\right)^{-\frac{n-p}{p}},
\]
the classical manifold of extremals is given by
\begin{equation*}
\mathcal M_{\rm Sob}
:=
\left\{
c\,\lambda^{\frac{n-p}{p}}U(\lambda(\cdot-x_0)):
c\in\mathbb R,\ \lambda>0,\ x_0\in\mathbb R^n
\right\}.
\end{equation*}

The quantitative stability problem for sharp Sobolev-type inequalities has its
roots in the work of Br\'ezis--Lieb~\cite{BL85} and
Bianchi--Egnell~\cite{BE91}. For the classical Sobolev inequality
\eqref{eq:intro-sobolev}, one asks whether the deficit
\begin{equation}
\label{def-deficit-classical-p}
\delta_{\rm Sob}(u)
:=
\|\nabla u\|_{L^p(\mathbb R^n)}^p
-
S_{n,p}^p\|u\|_{L^{p^*}(\mathbb R^n)}^p
\end{equation}
controls the distance from \(u\) to the extremal manifold
\(\mathcal M_{\rm Sob}\). In the Hilbertian case \(p=2\), this was proved by
Bianchi--Egnell~\cite{BE91}. For general \(1<p<n\),
Cianchi--Fusco--Maggi--Pratelli obtained a quantitative Sobolev inequality
with a remainder term involving a distance to the family of
extremals~\cite{CFMP09}. Strong gradient-distance stability for \(p\ge2\) was
proved by Figalli--Neumayer~\cite{FN19}, extended in strong form to the full
range \(1<p<n\) by Neumayer~\cite{Neumayer20}, and the sharp gradient
stability exponent in the full range was established by
Figalli--Zhang~\cite{FZ22}. More recently,
Dolbeault--Esteban--Figalli--Frank--Loss proved sharp stability estimates with
explicit constants and optimal dimensional dependence~\cite{DEFFL25}.

A parallel quantitative theory has been developed for several related critical
Sobolev-type inequalities. For \(BV\)-Sobolev and anisotropic Sobolev
inequalities, see Fusco--Maggi--Pratelli~\cite{FMP07} and
Figalli--Maggi--Pratelli~\cite{FMP13}. Fractional Sobolev stability was proved
by Chen--Frank--Weth~\cite{CFW13}, and the sharp stability constant together
with its asymptotic behavior has been further investigated by K\"onig~\cite{Konig23}
and Chen--Lu--Tang~\cite{CLT25}. Stability results are also available for
Hardy--Littlewood--Sobolev inequalities, Gagliardo--Nirenberg--Sobolev
inequalities, and trace inequalities; see, for instance,
\cite{CLT24,Carlen25,ZZ25,Nguyen19,CF13,Seuffert16,ZZZ25,Ho22}.

Affine Sobolev theory provides a genuinely stronger and more geometric
refinement of \eqref{eq:intro-sobolev}. The endpoint affine Sobolev--Zhang
inequality was established by Zhang~\cite{Zhang99}, and the sharp affine
\(L^p\)-Sobolev inequality for \(1<p<n\) was proved by
Lutwak--Yang--Zhang~\cite{LYZ02}, building on the \(L_p\) affine isoperimetric
theory developed in~\cite{LYZ00}. To state the inequality in a normalization
compatible with \eqref{eq:intro-sobolev}, we set, for
\(u\in\dot W^{1,p}(\mathbb R^n)\) and \(\xi\in\mathbb S^{n-1}\),
\begin{equation}
\label{eq:def-Axi}
A_\xi(u)
:=
\int_{\mathbb R^n}|\partial_\xi u|^p\,dx,
\qquad
\partial_\xi u:=\xi\cdot\nabla u .
\end{equation}
Let
\[
\langle g\rangle_{\mathbb S^{n-1}}
:=
\frac1{|\mathbb S^{n-1}|}
\int_{\mathbb S^{n-1}}g(\xi)\,d\sigma(\xi),
\qquad
m_{n,p}
:=
\left\langle |\omega_1|^p\right\rangle_{\omega\in\mathbb S^{n-1}} .
\]
We define the normalized affine \(L^p\)-energy by
\begin{equation}
\label{eq:intro-affine-energy-p}
\mathcal E(u)
:=
m_{n,p}^{-1/p}
\left\langle
A_\xi(u)^{-n/p}
\right\rangle_{\xi\in\mathbb S^{n-1}}^{-1/n}.
\end{equation}
With this convention, the sharp affine \(L^p\)-Sobolev inequality of
Lutwak--Yang--Zhang~\cite{LYZ02} takes the form
\begin{equation}
\label{eq:intro-affine-Lp}
S_{n,p}^p
\|u\|_{L^{p^*}(\mathbb R^n)}^p
\le
\mathcal E(u)^p
\le
\|\nabla u\|_{L^p(\mathbb R^n)}^p .
\end{equation}
Thus the affine energy refines the classical Dirichlet energy and, in
particular, recovers the sharp Sobolev inequality \eqref{eq:intro-sobolev} as a consequence. The
normalization in \eqref{eq:intro-affine-energy-p} is chosen so that
\(\mathcal E(u)^p=\|\nabla u\|_{L^p(\mathbb R^n)}^p\) for every
radial \(u\). In particular, the Talenti bubble \(U\) is an equality case in
both inequalities in \eqref{eq:intro-affine-Lp}. The affine inequality is
invariant under volume-preserving affine transformations, whereas the
classical Sobolev inequality is invariant only under Euclidean motions and
dilations. Accordingly, the equality cases in \eqref{eq:intro-affine-Lp} are
precisely the affine images of the Talenti bubbles~\cite{LYZ02}. This leads
to the affine extremal cone
\begin{equation*}
\mathcal M_{\rm aff}
:=
\left\{
c\,\lambda^{\frac{n-p}{p}}
U\bigl(\lambda A(x-x_0)\bigr):
c\in\mathbb R,\ \lambda>0,\ x_0\in\mathbb R^n,\ A\in SL(n)
\right\}.
\end{equation*}

Subsequent work has developed affine Sobolev theory in several directions.
Haberl--Schuster obtained asymmetric affine \(L^p\)-Sobolev
inequalities~\cite{HS09}, and Haberl--Schuster--Xiao established an
asymmetric affine P\'olya--Szeg\H{o} principle with applications to affine
logarithmic Sobolev inequalities~\cite{HSX12}.
Cianchi--Lutwak--Yang--Zhang proved affine Moser--Trudinger and
Morrey--Sobolev inequalities~\cite{CLYZ09}. Wang extended the affine
Sobolev--Zhang inequality to \(BV(\mathbb R^n)\) and characterized equality
cases~\cite{Wang12}. A different approach, based on the \(L^p\)
Busemann--Petty centroid inequality, was developed by
Haddad--Jimenez--Montenegro, yielding sharp affine Sobolev, log-Sobolev, and
Gagliardo--Nirenberg inequalities together with equality
cases~\cite{HJM16}; see also the weighted affine Sobolev inequalities
in~\cite{HJM19}. Affine trace and fractional variants have likewise been
studied, for example by De N\'apoli--Haddad--Jimenez--Montenegro~\cite{DHJM18}
and Haddad--Ludwig~\cite{HL24}.

Quantitative stability in the affine setting is considerably less developed.
Wang proved equality and stability results for the affine
P\'olya--Szeg\H{o} principle~\cite{Wang13}, and Nguyen obtained a stability
version of the affine Sobolev inequality in the \(BV\) framework via a new
approach to affine P\'olya--Szeg\H{o} symmetrization~\cite{Nguyen16}. More
recently, sharp stability for the affine fractional \(L^2\)-Sobolev
inequality was studied by Fan--Li--Zhang~\cite{FLZ26-1}.
The purpose of this paper is to prove a sharp \(W^{1,p}\)-gradient-distance
stability estimate for the affine \(L^p\)-Sobolev inequality \eqref{eq:intro-affine-Lp}, for
\(2\le p<n\).

For \(\lambda>0\), \(A\in SL(n)\), and \(x_0\in\mathbb R^n\), we write
\[
(T_{\lambda A,x_0}u)(x)
=
\lambda^{-\frac{n-p}{p}}
u\bigl(\lambda^{-1}A^{-1}x+x_0\bigr).
\]

Our main result is the following sharp affine stability theorem.

\begin{theorem}
\label{thm:intro-main-affine-p}
Let \(2\le p<n\). There exists a constant \(c_{n,p}>0\) such that every
nonzero \(u\in\dot W^{1,p}(\mathbb R^n)\) satisfies
\begin{equation}
\label{eq:intro-main-affine-p}
\frac{\mathcal E(u)}
{S_{n,p}\|u\|_{L^{p^*}(\mathbb R^n)}}-1
\ge
c_{n,p}
\left(  \inf_{\substack{a\in\mathbb R\\ A\in SL(n),\,\lambda>0,\,
x_0\in\mathbb R^n}}
\frac{
\left\|
\nabla\left(T_{\lambda A,x_0}u-aU\right)
\right\|_{L^p(\mathbb R^n)}
}{
\left\|
\nabla\left(T_{\lambda A,x_0}u\right)
\right\|_{L^p(\mathbb R^n)}}\right)^{p}.
\end{equation}
The exponent \(p\) is sharp.
\end{theorem}

\subsection{Strategy of the proof}

Our argument is modeled on the now classical stability scheme developed by
Bianchi--Egnell~\cite{BE91}, Figalli--Neumayer~\cite{FN19}, and
Figalli--Zhang~\cite{FZ22}. In the affine setting, however, two genuinely new
features enter. First, the affine energy is not given by a single Dirichlet
integral, but rather by a negative mean of the directional energies
\[
A_\xi(u)=\int_{\mathbb R^n}|\partial_\xi u|^p\,dx,
\qquad
\mathcal E(u)^p
=
m_{n,p}^{-1/p}
\left\langle
A_\xi(u)^{-n/p}
\right\rangle_{\mathbb S^{n-1}}^{-p/n}.
\]
Second, the second variation of this negative mean produces a genuinely
affine correction \(\mathcal R_p\), and this enlarges the kernel of the
classical Sobolev Hessian.

We begin with the local analysis near the normalized bubble. Thus we consider
\[
u=U+\varepsilon\phi,
\qquad
\|\nabla\phi\|_{L^p}=1,
\qquad
\phi\perp_{L^2(U^{p^*-2})}T_U\mathcal M_{\rm aff},
\qquad
0<\varepsilon\ll1 .
\]
The objective at this stage is to prove the lower bound
\begin{equation}
\label{eq:intro-locally estimate-p}
\delta_{\rm aff}(U+\varepsilon\phi)
\ge c\varepsilon^p .
\end{equation}

The pointwise nonlinear expansion from Figalli--Zhang~\cite{FZ22} yields, for
each fixed \(\xi\in\mathbb S^{n-1}\),
\[
A_\xi(U+\varepsilon\phi)
\ge
A_0+
\varepsilon L_\xi(\phi)+
\varepsilon^2(1-\kappa)
\mathfrak B_{\xi,\varepsilon}(\phi)+
c_\kappa\varepsilon^p A_\xi(\phi).
\]
Passing this lower bound through the negative mean yields the finite-scale
quadratic form; see Lemma~\ref{lem:affine-energy-lower-scale},
\begin{equation}
\label{eq:intro-finite-scale}
\mathcal N_{\rm aff}^{\varepsilon,\kappa}(\phi)
=
\alpha_{n,p}(1-\kappa)
\left\langle
\mathfrak B_{\xi,\varepsilon}(\phi)
\right\rangle_{\mathbb S^{n-1}}
-
\alpha_{n,p}(1+\kappa)
\frac{\tau+1}{2A_0}
\operatorname{Var}_{\xi}(L_\xi(\phi)),
\quad
\tau=\frac np .
\end{equation}
The variance term is precisely the new affine contribution, absent in the
classical theory. In the limit \(\varepsilon\downarrow0\), \eqref{eq:intro-finite-scale} converges to the affine Hessian
\[
Q_{\rm aff,p}(\phi)
=
Q_{\rm Sob,p}(\phi)
-
\alpha_{n,p}
\frac{\tau+1}{pA_0}
\operatorname{Var}_{\xi}(L_\xi(\phi)).
\]

A central step in the proof is the identification of the affine kernel:
\[
\ker Q_{\rm aff,p}
=
T_U\mathcal M_{\rm aff}
=
\operatorname{span}
\left\{
U,\ Z_0,\ \partial_{x_1}U,\dots,\partial_{x_n}U,\ x\cdot B\nabla U
\right\},
\qquad
B=B^T,\quad \operatorname{tr}B=0 .
\]
This is achieved by decomposing perturbations into physical spherical harmonic
sectors. The radial and translation sectors are governed by the classical
nondegeneracy theory for the critical \(p\)-Sobolev bubble
\cite{PV21,FN19,FZ22}. The genuinely affine part is the degree-two
trace-free sector, where the variance correction creates precisely the
additional Jacobi fields \(x\cdot B\nabla U\). This mechanism is parallel to
the affine fractional Hilbertian case studied in \cite{FLZ26-1}.

Once the affine kernel has been identified, one obtains a spectral gap; see Proposition~\ref{prop:calibrated-affine-spectral}, namely
\[
\mathcal N_{\rm aff}^{\varepsilon,\kappa}(\phi)
\ge
\frac p2
\left((p^*-1)\Lambda+\eta_0\right)
\int_{\mathbb R^n}U^{p^*-2}\phi^2\,dx,
\qquad  \phi\perp_{L^2(\mathbb R^n,U^{p^*-2}dx)}
T_U\mathcal M_{\rm aff}.
\]
Subtracting the corresponding
\(L^{p^*}\)-expansion then yields \eqref{eq:intro-locally estimate-p}.
This establishes the local affine stability estimate.

The global estimate is then obtained by contradiction. A John-ellipsoid affine
normalization \cite[Theorem~10.12.2]{Schneider14} rules out affine escape and controls the classical Dirichlet
energy by the affine energy. One then applies the critical Sobolev profile
decomposition. Combined with directional decoupling and the reverse Minkowski
inequality for the negative mean, this excludes splitting and shows that every
affine near-extremizing sequence converges, after affine transformations and
multiplication by constants, to an affine Talenti bubble. The local estimate
then upgrades this qualitative compactness statement to the full global
stability inequality.

Finally, we prove that the exponent \(p\) is optimal. For this we test the
deficit against translated bumps
\[
0\ne\zeta\in C_c^\infty(\mathbb R^n)
\qquad
\zeta_R(x):=\zeta(x-Re_1),
\qquad
u_{\varepsilon,R}:=U+\varepsilon\zeta_R .
\]

The paper is organized as follows. Section~\ref{subsec:second-variation-affine-p}
is devoted to the second variation of the affine energy at the Talenti bubble.
There we derive both the limiting affine Hessian and the finite-scale
quadratic form governing the nonlinear analysis, and we reduce the local
stability problem to the calibrated estimate of
Proposition~\ref{prop:calibrated-affine-spectral}. Section~\ref{sec:pro-spectral} is devoted
to the proof of Proposition~\ref{prop:calibrated-affine-spectral}. In Section~\ref{sec:global-affine-stability}, we establish
qualitative compactness for affine near-extremizers and prove Theorem~\ref{thm:intro-main-affine-p}. Finally,
Section~\ref{sec:sector-decomposition-affine-p} contains the spherical
harmonic analysis of the affine Hessian, including the identification of its
kernel and the positivity of the higher even sectors.

\section{Second variation of the affine energy}
\label{subsec:second-variation-affine-p}

Let
\[
\tau=\frac np,
\qquad
p^*=\frac{np}{n-p},
\qquad
\Phi(a)=
\left\langle a(\xi)^{-\tau}\right\rangle_{\mathbb S^{n-1}}^{-1/\tau}.
\]
By \eqref{eq:def-Axi}, the affine energy admits the representation
\begin{equation}
\label{eq:affine-energy}
 \mathcal E(u)^p
=
\alpha_{n,p}\Phi(A_\xi(u)),
\qquad   \alpha_{n,p}=m_{n,p}^{-1},
\qquad m_{n,p}=
\left\langle |\omega_1|^p\right\rangle_{\omega\in\mathbb S^{n-1}}.
\end{equation}
We also record the elementary subadditivity property
\begin{equation}
\label{eq:sub-add}
\Phi(a+b)\ge \Phi(a)+\Phi(b),
\end{equation}
valid for all nonnegative \(a,b\).

Suppose now that \(u=u(r)\) is radial. Writing \(x=r\theta\), one has
\[
\partial_\xi u(x)=u'(r)(\theta\cdot\xi).
\]
Hence, by rotation invariance, \eqref{eq:def-Axi} reduces to
\begin{equation*}
A_\xi(u)
=
\int_0^\infty |u'(r)|^p r^{n-1}\,dr
\int_{\mathbb S^{n-1}}|\theta\cdot\xi|^p\,d\sigma(\theta)=
m_{n,p}\|\nabla u\|_{L^p(\mathbb R^n)}^p .
\end{equation*}
In particular, \(A_\xi(u)\) is independent of \(\xi\), and therefore
\[
\Phi(A_\xi(u))
=
m_{n,p}\|\nabla u\|_{L^p(\mathbb R^n)}^p .
\]

Since \(U\) is radial, \(A_\xi(U)\) is likewise independent of \(\xi\); we denote
this value by
\(A_\xi(U)=A_0\).
Let \(u_t=U+t\phi\), where \(\phi\in C_c^\infty(\mathbb R^n)\). By
\eqref{eq:def-Axi},
\[
A_\xi(u_t)
=
A_0+tL_\xi(\phi)+t^2B_\xi(\phi)+o(t^2),
\]
where
\begin{equation}
\label{eq:def-Lxi}
L_\xi(\phi) :=
p\int_{\mathbb R^n}
|\partial_\xi U|^{p-2}\partial_\xi U\,\partial_\xi\phi\,dx,
\end{equation}
and
\[
B_\xi(\phi) :=
\frac{p(p-1)}2
\int_{\mathbb R^n}
|\partial_\xi U|^{p-2}|\partial_\xi\phi|^2\,dx .
\]
Expanding
\[
\mathcal E(u)^p
=
\alpha_{n,p}
\Phi(A_\xi(u)),
\qquad
\Phi(a)=\left\langle a^{-\tau}\right\rangle_{\mathbb S^{n-1}}^{-1/\tau},
\]
around the constant function \(A_0\), we obtain
\begin{equation}
\label{eq:affine-energy-expansion-p-clean}
\begin{aligned}
\mathcal E(U+t\phi)^p
&=
\alpha_{n,p}A_0
+ t\alpha_{n,p}
\left\langle L_\xi(\phi)\right\rangle_{\mathbb S^{n-1}}
\\
&\quad
+ t^2\alpha_{n,p}
\left[
\left\langle B_\xi(\phi)\right\rangle_{\mathbb S^{n-1}}
-
\frac{\tau+1}{2A_0}
\operatorname{Var}_{\xi}
\bigl(L_\xi(\phi)\bigr)
\right]
+ o(t^2),
\end{aligned}
\end{equation}
where
\begin{equation}
\label{eq:def-varxi}
 \operatorname{Var}_{\xi}(L_\xi(\phi))
:=
\left\langle L_\xi(\phi)^2\right\rangle_{\mathbb S^{n-1}}
-
\left\langle L_\xi(\phi)\right\rangle_{\mathbb S^{n-1}}^2.
\end{equation}

We next collect the angular identities needed below. Recalling that
\(m_{n,p}=\left\langle |\omega_1|^p\right\rangle_{\omega\in\mathbb S^{n-1}}\),
for every \(\theta\in\mathbb S^{n-1}\) and every \(\eta\in\mathbb R^n\), one has
\begin{equation}
\label{eq:angular-linear-identity-p}
\left\langle
|\theta\cdot\xi|^{p-2}(\theta\cdot\xi)(\eta\cdot\xi)
\right\rangle_{\xi}
=
m_{n,p}\,\theta\cdot\eta,
\end{equation}
and
\begin{equation}
\label{eq:angular-quadratic-identity-p}
\left\langle
|\theta\cdot\xi|^{p-2}(\eta\cdot\xi)^2
\right\rangle_{\xi}
=
\frac{m_{n,p}}{p-1}
\left(
|\eta|^2+(p-2)(\theta\cdot\eta)^2
\right).
\end{equation}
Indeed, by rotation invariance, it is enough to take \(\theta=e_1\). The two
identities then follow by differentiating once and twice the identity
\begin{equation}
\label{eq:identify-z}
\left\langle |\xi\cdot z|^p\right\rangle_{\xi}
=
m_{n,p}|z|^p
\end{equation}
at \(z=e_1\).

Since \(U=U(r)\), writing \(x=r\theta\), we have
\begin{equation}
\label{eq:xiU}
\partial_\xi U=U'(r)(\theta\cdot\xi).
\end{equation}
Using \eqref{eq:angular-linear-identity-p}, together with
\[
-\Delta_pU=\Lambda U^{p^*-1},
\qquad \Lambda
=
n\left(\frac{n-p}{p-1}\right)^{p-1},
\]
we infer
\begin{equation}
\label{eq:average-Lxi-p}
\alpha_{n,p}
\left\langle L_\xi(\phi)\right\rangle_{\mathbb S^{n-1}}
=
p\int_{\mathbb R^n}
|\nabla U|^{p-2}\nabla U\cdot\nabla\phi\,dx
=
p\Lambda\int_{\mathbb R^n}U^{p^*-1}\phi\,dx .
\end{equation}
Similarly, \eqref{eq:angular-quadratic-identity-p} yields
\begin{equation}
\label{eq:average-Bxi-dirichlet-p}
\frac{2}{p}\alpha_{n,p}
\left\langle B_\xi(\phi)\right\rangle_{\mathbb S^{n-1}}
=
\int_{\mathbb R^n}
|\nabla U|^{p-2}
\left(
|\nabla\phi|^2
+(p-2)
\frac{(\nabla U\cdot\nabla\phi)^2}{|\nabla U|^2}
\right)\,dx.
\end{equation}

We now recall the quadratic form associated with the classical Sobolev deficit.
Recall \eqref{def-deficit-classical-p},
\[
\delta_{\rm Sob}(u)
:=
\|\nabla u\|_{L^p(\mathbb R^n)}^p
-
S_{n,p}^p\|u\|_{L^{p^*}(\mathbb R^n)}^p ,
\qquad  S_{n,p}^p=\Lambda \left(\int_{\mathbb R^n}U^{p^*}\,dx\right)^{1-\frac{p}{p^*}}.
\]
The Hessian \(Q_{\rm Sob,p}\) is defined by
\[
\delta_{\rm Sob}(U+t\phi)
=
\frac p2t^2Q_{\rm Sob,p}(\phi)+o(t^2).
\]
A straightforward second-order expansion gives
\begin{equation}
\label{eq:classical-p-homogeneous-Hessian}
\begin{aligned}
Q_{\rm Sob,p}(\phi)
&=
\int_{\mathbb R^n}
|\nabla U|^{p-2}
\left(
|\nabla\phi|^2
+(p-2)
\frac{(\nabla U\cdot\nabla\phi)^2}{|\nabla U|^2}
\right)\,dx\\
&\quad-
(p^*-1)\Lambda
\int_{\mathbb R^n}U^{p^*-2}\phi^2\,dx
\\
&\quad+(p^*-p)\Lambda \left(\int_{\mathbb R^n}U^{p^*}\,dx\right)^{-1}
\left(
\int_{\mathbb R^n}U^{p^*-1}\phi\,dx
\right)^2 .
\end{aligned}
\end{equation}

Finally, define the affine deficit by
\begin{equation}
\label{eq:def-affine-deficit}
\delta_{\rm aff}(u)
:=
\mathcal E(u)^p
-
S_{n,p}^p\|u\|_{L^{p^*}(\mathbb R^n)}^p ,
\end{equation}
and define \(Q_{\rm aff,p}\) through
\[
\delta_{\rm aff}(U+t\phi)
=
\frac p2t^2Q_{\rm aff,p}(\phi)+o(t^2).
\]
Combining \eqref{eq:affine-energy-expansion-p-clean},
\eqref{eq:average-Bxi-dirichlet-p}, and
\eqref{eq:classical-p-homogeneous-Hessian}, we arrive at
\begin{equation}
\label{eq:Q-aff-decomp-p}
Q_{\rm aff,p}(\phi)
=
Q_{\rm Sob,p}(\phi)-\mathcal R_p(\phi),
\end{equation}
where
\begin{equation}
\label{eq:R-p-variance}
\mathcal R_p(\phi)
:=
\alpha_{n,p}
\frac{\tau+1}{pA_0}
\operatorname{Var}_{\xi}
\bigl(L_\xi(\phi)\bigr).
\end{equation}

\subsection{Nonlinear estimate}

In this subsection we reduce the nonlinear lower bound for the affine deficit \eqref{eq:def-affine-deficit} to
a calibrated spectral estimate for the finite-scale quadratic form
\(\mathcal N_{\rm aff}^{\varepsilon,\kappa}\).

\begin{lemma}
\label{lem:affine-energy-lower-scale}
Fix \(0<\kappa<1\). There exist \(\varepsilon_0>0\) and \(c_1>0\) such that
the following holds. If
\[
0<\varepsilon\le\varepsilon_0,
\qquad
\|\nabla\phi\|_{L^p(\mathbb R^n)}=1,
\]
then
\begin{equation}
\label{eq:affine-energy-lower}
\mathcal E(U+\varepsilon\phi)^p
\ge
\mathcal E(U)^p+
\varepsilon\alpha_{n,p}
\left\langle L_\xi(\phi)\right\rangle_{\mathbb S^{n-1}}
+
\varepsilon^2
\mathcal N_{\rm aff}^{\varepsilon,\kappa}(\phi)
+
c_1\varepsilon^p
\left\langle A_\xi(\phi)\right\rangle_{\mathbb S^{n-1}},
\end{equation}
where
\begin{equation}
\label{eq:def-N-aff-FZ}
\begin{aligned}
\mathcal N_{\rm aff}^{\varepsilon,\kappa}(\phi)
:=&\,
\alpha_{n,p}(1-\kappa)
\left\langle
\mathfrak B_{\xi,\varepsilon}(\phi)
\right\rangle_{\mathbb S^{n-1}}
-
\alpha_{n,p}(1+\kappa)
\frac{\tau+1}{2A_0}
\operatorname{Var}_{\xi\in\mathbb S^{n-1}}
\bigl(L_\xi(\phi)\bigr).
\end{aligned}
\end{equation}
\end{lemma}

\begin{proof}
We start from the one-directional lower bound provided by
\cite[Lemma~2.1~(ii)]{FZ22}. For every \(\xi\in\mathbb S^{n-1}\) and every
\(0<\kappa_*<1\),
\begin{equation}
\label{eq:directional-energy-bound}
A_\xi(U+\varepsilon\phi)
\ge
A_0+\varepsilon L_\xi(\phi)
+
\varepsilon^2(1-\kappa_*)
\mathfrak B_{\xi,\varepsilon}(\phi)
+c_{\kappa_*}\varepsilon^pA_\xi(\phi),
\end{equation}
where
\begin{equation}
\label{eq:def-Bxi}
\begin{aligned}
\mathfrak B_{\xi,\varepsilon}(\phi)
&:=
\frac p2
\int_{\mathbb R^n}
|\partial_\xi U|^{p-2}|\partial_\xi\phi|^2\,dx\\
&\quad+
\frac{p(p-2)}2
\int_{\mathbb R^n}
|w_{\xi,\varepsilon}|^{p-2}
\left(
\frac{|\partial_\xi U+\varepsilon \partial_\xi\phi|-|\partial_\xi U|}{\varepsilon}
\right)^2dx,
\end{aligned}
\end{equation}
and
\begin{equation}
\label{eq:def-w}
w_{\xi,\varepsilon}
:=
\begin{cases}
\partial_\xi U,
& |\partial_\xi U|<|\partial_\xi U+\varepsilon \partial_\xi\phi|,\\[0.4em]
\left(\dfrac{|\partial_\xi U+\varepsilon \partial_\xi\phi|}{|\partial_\xi U|}\right)^{\frac1{p-2}}
(\partial_\xi U+\varepsilon \partial_\xi\phi),
& |\partial_\xi U+\varepsilon \partial_\xi\phi|\le |\partial_\xi U|,\ \partial_\xi U\ne0,\\[0.4em]
0,
& \partial_\xi U=0 .
\end{cases}
\end{equation}

We next record the uniform bounds needed to insert
\eqref{eq:directional-energy-bound} into the nonlinear expansion of the
negative mean. Since \(\|\nabla\phi\|_{L^p(\mathbb R^n)}=1\), we have,
uniformly in \(\xi\in\mathbb S^{n-1}\),
\[
A_\xi(\phi)=
\int_{\mathbb R^n}|\partial_\xi\phi|^p\,dx
\le
\int_{\mathbb R^n}|\nabla\phi|^p\,dx
=1 .
\]
Moreover, by H\"older's inequality,
\[
|L_\xi(\phi)|=
p\left|
\int_{\mathbb R^n}
|\partial_\xi U|^{p-2}\partial_\xi U\,\partial_\xi\phi\,dx
\right|
\le p
\left(\int_{\mathbb R^n}|\partial_\xi U|^p\,dx\right)^{\frac{p-1}{p}}
\left(\int_{\mathbb R^n}|\partial_\xi\phi|^p\,dx\right)^{\frac1p}
\le C .
\]
Likewise, \eqref{eq:def-Bxi}--\eqref{eq:def-w} give
\[
\mathfrak B_{\xi,\varepsilon}(\phi)
\le C
\int_{\mathbb R^n}
|\partial_\xi U|^{p-2}|\partial_\xi\phi|^2\,dx
\le C
\qquad
\text{uniformly in } \xi\in\mathbb S^{n-1},\ 0<\varepsilon\le1 .
\]

Set
\[
h_\varepsilon(\xi):=
\varepsilon L_\xi(\phi)
+
\varepsilon^2(1-\kappa_*)
\mathfrak B_{\xi,\varepsilon}(\phi)
+
c_{\kappa_*}\varepsilon^pA_\xi(\phi).
\]
Upon decreasing \(\varepsilon_0\) if necessary, the preceding bounds ensure
that \(A_0+h_\varepsilon>0\) on \(\mathbb S^{n-1}\). Therefore,
\eqref{eq:directional-energy-bound}, the monotonicity of the negative mean, and
the representation formula \eqref{eq:affine-energy} imply
\[
\mathcal E(U+\varepsilon\phi)^p
\ge
\alpha_{n,p}\Phi(A_0+h_\varepsilon).
\]
Since \(\|h_\varepsilon\|_{L^\infty(\mathbb S^{n-1})}\le C\varepsilon\), the
Taylor expansion of \(\Phi\) at the constant state \(A_0\) yields, after
possibly decreasing \(\varepsilon_0\) once more,
\begin{equation}
\label{eq:negative-mean-concrete-short}
\Phi(A_0+h_\varepsilon)
\ge
A_0+\langle h_\varepsilon\rangle_{\mathbb S^{n-1}}
-
(1+\gamma)\frac{\tau+1}{2A_0}
\operatorname{Var}(h_\varepsilon).
\end{equation}

We now separate the linear part from the higher-order remainder. Write
\[
h_\varepsilon
=
\varepsilon L_\xi(\phi)+s_\varepsilon,
\qquad
s_\varepsilon
=
\varepsilon^2(1-\kappa_*)
\mathfrak B_{\xi,\varepsilon}(\phi)+
c_{\kappa_*}\varepsilon^p A_\xi(\phi).
\]
Then \(s_\varepsilon\ge0\) and
\(\|s_\varepsilon\|_{L^\infty(\mathbb S^{n-1})}\le C\varepsilon^2\). Hence
\[
\operatorname{Var}(s_\varepsilon)
\le
\langle s_\varepsilon^2\rangle_{\mathbb S^{n-1}}
\le
\|s_\varepsilon\|_{L^\infty(\mathbb S^{n-1})}
\langle s_\varepsilon\rangle_{\mathbb S^{n-1}}
\le
C\varepsilon^2
\langle s_\varepsilon\rangle_{\mathbb S^{n-1}} .
\]
It follows that
\begin{equation}
\label{eq:varh-control}
\begin{aligned}
\operatorname{Var}(h_{\varepsilon})
&=
\left\langle
\left[
\varepsilon
\left(
L_{\xi}-\langle L_{\xi}\rangle_{\mathbb S^{n-1}}
\right)+s_{\varepsilon}-
\langle s_{\varepsilon}\rangle_{\mathbb S^{n-1}}
\right]^2
\right\rangle_{\mathbb S^{n-1}}\\
&\le
(1+\gamma)\varepsilon^2
\operatorname{Var}_{\xi}(L_\xi(\phi))
+
C_\gamma \operatorname{Var}(s_\varepsilon)\\
&\le(1+\gamma)\varepsilon^2
\operatorname{Var}_{\xi}(L_\xi(\phi))
+
C_\gamma\varepsilon^2
\langle s_\varepsilon\rangle_{\mathbb S^{n-1}} .
\end{aligned}
\end{equation}

Substituting \eqref{eq:varh-control} into
\eqref{eq:negative-mean-concrete-short}, and absorbing the term
\(C_\gamma\varepsilon^2\langle s_\varepsilon\rangle\) by taking
\(\varepsilon_0\) sufficiently small, we obtain
\[
\begin{aligned}
\Phi(A_0+h_\varepsilon)
&\ge A_0+
\varepsilon\left\langle L_\xi(\phi)\right\rangle_{\mathbb S^{n-1}}+
\langle s_\varepsilon\rangle_{\mathbb S^{n-1}}
-
C_\gamma\varepsilon^2\langle s_\varepsilon\rangle_{\mathbb S^{n-1}} \\
&\qquad -
(1+\gamma)^2
\frac{\tau+1}{2A_0}
\varepsilon^2
\operatorname{Var}_{\xi}(L_\xi(\phi))\\
&\ge
A_0+
\varepsilon
\left\langle L_\xi(\phi)\right\rangle_{\mathbb S^{n-1}}
+
(1-\gamma)\langle s_\varepsilon\rangle_{\mathbb S^{n-1}}-
(1+\gamma)^2
\frac{\tau+1}{2A_0}
\varepsilon^2
\operatorname{Var}_{\xi}(L_\xi(\phi))\\
&\ge
A_0+
\varepsilon
\left\langle L_\xi(\phi)\right\rangle_{\mathbb S^{n-1}}
+
(1-\gamma)\varepsilon^2(1-\kappa_*)
\left\langle
\mathfrak B_{\xi,\varepsilon}(\phi)
\right\rangle_{\mathbb S^{n-1}}
\\
&\quad-
(1+\gamma)^2
\frac{\tau+1}{2A_0}
\varepsilon^2
\operatorname{Var}_{\xi}(L_\xi(\phi))
+
c_1\varepsilon^p
\left\langle A_\xi(\phi)\right\rangle_{\mathbb S^{n-1}} .
\end{aligned}
\]
Multiplying by \(\alpha_{n,p}\), invoking
\(\alpha_{n,p}A_0=\mathcal E(U)^p\),
and choosing \(\gamma,\kappa_*>0\) so that
\((1-\gamma)(1-\kappa_*)\ge1-\kappa\), \((1+\gamma)^2\le1+\kappa\),
we arrive precisely at \eqref{eq:affine-energy-lower}.
\end{proof}

We next subtract the \(L^{p^*}(\mathbb R^n)\)-term. By
\cite[Lemma~3.2]{FN19}, for every \(\eta_0>0\),
\begin{equation}
\label{eq:Lq-upper}
\begin{aligned}
S_{n,p}^p
\|U+\varepsilon\phi\|_{L^{p^*}(\mathbb R^n)}^p
&\le
S_{n,p}^p
\|U\|_{L^{p^*}(\mathbb R^n)}^p+
\varepsilon p\Lambda
\int_{\mathbb R^n}U^{p^*-1}\phi\,dx\\
&\quad
+
\varepsilon^2
\frac p2
\left((p^*-1)\Lambda+\frac{\eta_0}{2}\right)
\int_{\mathbb R^n}U^{p^*-2}\phi^2\,dx
\\
&\quad+
C\varepsilon^{p^*}
\int_{\mathbb R^n}|\phi|^{p^*}\,dx .
\end{aligned}
\end{equation}
Combining \eqref{eq:affine-energy-lower}, \eqref{eq:Lq-upper}, and
\eqref{eq:average-Lxi-p}, we obtain
\begin{equation}
\label{eq:affine-deficit-lower-scale}
\begin{aligned}
\delta_{\rm aff}(U+\varepsilon\phi)
&\ge
\varepsilon^2
\left[
\mathcal N_{\rm aff}^{\varepsilon,\kappa}(\phi)
-
\frac p2
\left((p^*-1)\Lambda+\frac{\eta_0}{2}\right)
\int_{\mathbb R^n}U^{p^*-2}\phi^2\,dx
\right]\\
&\quad+
c_1\varepsilon^p
\left\langle A_\xi(\phi)\right\rangle_{\mathbb S^{n-1}}-
C\varepsilon^{p^*}
\int_{\mathbb R^n}|\phi|^{p^*}\,dx .
\end{aligned}
\end{equation}

Thus the nonlinear part of the argument is reduced to establishing a strictly
positive lower bound for the bracketed term in
\eqref{eq:affine-deficit-lower-scale}. The precise statement needed for the
stability proof is the following.

\begin{proposition}
\label{prop:calibrated-affine-spectral}
Let \(2\le p<n\). There exist constants
\(\kappa_0>0\), \(\varepsilon_0>0\), and \(\eta_0>0\) such that, for every
\(0<\kappa\le\kappa_0\), every \(0<\varepsilon\le\varepsilon_0\), and every
\(\phi\in\dot W^{1,p}(\mathbb R^n)\) satisfying
\[
\|\nabla\phi\|_{L^p(\mathbb R^n)}=1,
\qquad
\phi\perp_{L^2(\mathbb R^n,U^{p^*-2}dx)}
T_U\mathcal M_{\rm aff},
\]
one has
\begin{equation}
\label{eq:calibrated-affine-spectral}
\mathcal N_{\rm aff}^{\varepsilon,\kappa}(\phi)
\ge
\frac p2
\bigl((p^*-1)\Lambda+\eta_0\bigr)
\int_{\mathbb R^n}U^{p^*-2}\phi^2\,dx .
\end{equation}
\end{proposition}

Indeed, assuming Proposition~\ref{prop:calibrated-affine-spectral} holds, estimate
\eqref{eq:affine-deficit-lower-scale} immediately yields
\begin{equation}
\label{eq:D-ge-2}
\delta_{\rm aff}(U+\varepsilon\phi)
\ge
\frac{p\eta_0}{4}\varepsilon^2
\int_{\mathbb R^n}U^{p^*-2}\phi^2\,dx
+
c_1\varepsilon^p
\left\langle A_\xi(\phi)\right\rangle_{\mathbb S^{n-1}}
-
C\varepsilon^{p^*}
\int_{\mathbb R^n}|\phi|^{p^*}\,dx .
\end{equation}

Since \(p^*>p\) and \(\|\nabla\phi\|_{L^p(\mathbb R^n)}=1\), the final term
is of higher order in \(\varepsilon\) and can therefore be absorbed once
\(\varepsilon\) is chosen sufficiently small.

\section{Proof of Proposition~\ref{prop:calibrated-affine-spectral}}
\label{sec:pro-spectral}

The argument proceeds in three steps. We first establish a genuine coercive gap
for the limiting affine Hessian \(Q_{\rm aff,p}\). We then prove the compactness
properties needed to pass from the finite-scale form
\(\mathcal N_{\rm aff}^{\varepsilon,\kappa}\) to the limiting quadratic form.
Finally, we conclude by contradiction: if
Proposition~\ref{prop:calibrated-affine-spectral} were false, one would obtain,
after normalization and passage to the limit, a nontrivial element in the null
space of \(Q_{\rm aff,p}\), contradicting the spectral gap.

\subsection{Spectral gap}

We now strengthen the kernel classification from
Section~\ref{sec:sector-decomposition-affine-p} to a coercive estimate. By Section~\ref{sec:sector-decomposition-affine-p},
\(T_U\mathcal \mathcal M_{\rm aff}
=
\ker Q_{\rm aff,p}\).
We work in the quadratic space
\[
\mathcal Z_U
:=
\overline{C_c^\infty(\mathbb R^n)}
^{\|\cdot\|_{\mathcal Z_U}},
\qquad
\|\phi\|_{\mathcal Z_U}^2
:=
\int_{\mathbb R^n}
|\nabla U|^{p-2}|\nabla\phi|^2\,dx
+
\int_{\mathbb R^n}
U^{p^*-2}\phi^2\,dx .
\]
We also introduce
\[
\|\phi\|_{\mathcal H_U}^2
:=
\int_{\mathbb R^n}
|\nabla U|^{p-2}|\nabla\phi|^2\,dx,
\qquad
\mathcal H_U
:=
\overline{C_c^\infty(\mathbb R^n)}
^{\|\cdot\|_{\mathcal H_U}} .
\]

By \cite[Corollary~6.2]{FN19}, equivalently
\cite[Proposition~3.2]{FZ22}, one has the compact weighted embedding
\begin{equation}
\label{eq:weighted-compact-embedding}
\mathcal H_U
\hookrightarrow\hookrightarrow
L^2\bigl(\mathbb R^n,U^{p^*-2}dx\bigr).
\end{equation}
The weighted Poincar\'e inequality \cite[Lemma~3.3]{FZ22} gives
\begin{equation}
\label{eq:weighted-Poincare}
     \int_{\mathbb R^n}
        U^{p^*-2}\phi^2\,dx
        \le
       c_0\|\phi\|_{\mathcal H_U}^2,
        \qquad
        \phi\in C_c^\infty(\mathbb R^n).
\end{equation}

\begin{proposition}
\label{prop:affine-spectral-gap}
There exists \(c_{\rm sg}>0\) such that, for every
\(\phi\in\mathcal H_U\) satisfying
\[
\phi\perp_{L^2(\mathbb R^n,U^{p^*-2}dx)}
T_U\mathcal \mathcal M_{\rm aff},
\]
one has
\[
Q_{\rm aff,p}(\phi)
\ge
c_{\rm sg}\|\phi\|_{\mathcal H_U}^2,
\]
\end{proposition}

\begin{proof}
Assume by contradiction that the stated coercive estimate fails. Then there
exists a sequence \(\phi_j\in\mathcal H_U\) such that
\[
\phi_j\perp_{L^2(\mathbb R^n,U^{p^*-2}dx)}
T_U\mathcal \mathcal M_{\rm aff},
\qquad
\|\phi_j\|_{\mathcal H_U}=1,
\qquad
Q_{\rm aff,p}(\phi_j)\to0 .
\]
By \eqref{eq:weighted-compact-embedding}, after passing to a subsequence,
\begin{equation}
\label{eq:weak-Hu}
\phi_j\rightharpoonup\phi
\quad\text{weakly in }\mathcal H_U,
\end{equation}
and
\begin{equation}
\label{eq:strong-l2}
\phi_j\to\phi
\quad\text{strongly in }
L^2(\mathbb R^n,U^{p^*-2}dx).
\end{equation}
The orthogonality condition is preserved in the limit, and therefore
\begin{equation}
\label{eq:orth-phi}
        \phi\perp_{L^2(\mathbb R^n,U^{p^*-2}dx)}
T_U\mathcal \mathcal M_{\rm aff}.
\end{equation}

We next pass to the variance term. Since the map
\( \xi\mapsto L_\xi\)
is continuous from \(\mathbb S^{n-1}\) into \(\mathcal H_U'\), the family
\(\{L_\xi:\xi\in\mathbb S^{n-1}\}
\subset \mathcal H_U'\)
is compact. Hence \eqref{eq:weak-Hu} implies
\[
\sup_{\xi\in\mathbb S^{n-1}}
|L_\xi(\phi_j)-L_\xi(\phi)|
\to0 .
\]
In particular,
\begin{equation}
\label{eq:cts-varL}
\operatorname{Var}_{\xi\in\mathbb S^{n-1}}
\bigl(L_\xi(\phi_j)\bigr)
\to
\operatorname{Var}_{\xi\in\mathbb S^{n-1}}
\bigl(L_\xi(\phi)\bigr).
\end{equation}

Using \eqref{eq:classical-p-homogeneous-Hessian},
\eqref{eq:Q-aff-decomp-p}, \eqref{eq:strong-l2}, and \eqref{eq:cts-varL}, we
find
\[
\begin{aligned}
Q_{\rm aff,p}(\phi)
&=
\int_{\mathbb R^n}
|\nabla U|^{p-2}
\left(
|\nabla\phi|^2
+
(p-2)
\frac{(\nabla U\cdot\nabla\phi)^2}{|\nabla U|^2}
\right)\,dx
\\
&\qquad
-
(p^*-1)\Lambda
\int_{\mathbb R^n}U^{p^*-2}\phi^2\,dx
-
\alpha_{n,p}
\frac{\tau+1}{pA_0}
\operatorname{Var}_{\xi\in\mathbb S^{n-1}}
\bigl(L_\xi(\phi)\bigr)\\
&\le
\liminf_{j\to\infty}
\int_{\mathbb R^n}
|\nabla U|^{p-2}
\left(
|\nabla\phi_j|^2
+
(p-2)
\frac{(\nabla U\cdot\nabla\phi_j)^2}{|\nabla U|^2}
\right)\,dx \\
&\qquad
-
(p^*-1)\Lambda\lim_{j\to\infty}
\int_{\mathbb R^n}U^{p^*-2}\phi_j^2\,dx
-
\alpha_{n,p}\lim_{j\to\infty}
\frac{\tau+1}{pA_0}
\operatorname{Var}_{\xi\in\mathbb S^{n-1}}
\bigl(L_\xi(\phi_j)\bigr)\\
&=\liminf_{j\to\infty}Q_{\rm aff,p}(\phi_j)=0.
\end{aligned}
\]
Since \(Q_{\rm aff,p}\ge0\), it follows that
\(Q_{\rm aff,p}(\phi)=0\).

By \eqref{eq:p-affine-kernel},
\(\phi\in T_U\mathcal \mathcal M_{\rm aff}\).
Together with \eqref{eq:orth-phi}, this yields \(\phi=0\). Hence
\eqref{eq:strong-l2} implies
\begin{equation}
\label{eq:strong-l2-zero}
\int_{\mathbb R^n}
U^{p^*-2}\phi_j^2\,dx
\to0 .
\end{equation}

Moreover, by \eqref{eq:classical-p-homogeneous-Hessian} and \eqref{eq:strong-l2-zero}, we obtain
\[
\int_{\mathbb R^n}
|\nabla U|^{p-2}|\nabla\phi_j|^2\,dx
\le
Q_{\rm aff,p}(\phi_j)+o(1)
\to0  .
\]
Therefore
\(\|\phi_j\|_{\mathcal H_U}\to0\),
which contradicts the normalization
\(\|\phi_j\|_{\mathcal H_U}=1\).
This contradiction proves the proposition.
\end{proof}

The spectral gap in Proposition~\ref{prop:affine-spectral-gap} concerns the
limiting quadratic form \(Q_{\rm aff,p}\). By contrast, the nonlinear lower
bound \eqref{eq:affine-deficit-lower-scale} involves the
\(\varepsilon\)-dependent quantity
\(\mathcal N_{\rm aff}^{\varepsilon,\kappa}\). Accordingly,
Proposition~\ref{prop:affine-spectral-gap} cannot be applied directly at the
finite scale.

The link between \(Q_{\rm aff,p}\) and
\(\mathcal N_{\rm aff}^{\varepsilon,\kappa}\) emerges only after passage to the
limit \(\varepsilon\downarrow0\). Indeed, \eqref{eq:affine-deficit-lower-scale}
formally suggests that
\[
Q_{\rm aff,p}(\phi)
\ge
\frac2p
\liminf_{\varepsilon\downarrow0}
\mathcal N_{\rm aff}^{\varepsilon,\kappa}(\phi)
-
(p^*-1)\Lambda
\int_{\mathbb R^n}U^{p^*-2}\phi^2\,dx .
\]
Thus the issue is to show that the coercivity of the limiting Hessian persists,
uniformly for the family of finite-scale forms
\(\mathcal N_{\rm aff}^{\varepsilon,\kappa}\), once \(\varepsilon\) is
sufficiently small. To pass from the limiting coercivity to the
\(\varepsilon\)-dependent setting, we require two compactness inputs. The first
is the continuity of
\(
\operatorname{Var}_{\xi\in\mathbb S^{n-1}} \bigl(L_\xi(\phi_j)\bigr)
\)
along weakly convergent sequences in \(\mathcal H_U\). The second is the lower
semicontinuity, as \(\varepsilon\downarrow0\), of the averaged term
\(\langle\mathfrak B_{\xi,\varepsilon}(\phi)\rangle\). The latter is the
content of the next lemma.

\begin{lemma}
\label{lem:averaged-affine-lsc}
Let \(2\le p<n\). Let
\[
\varepsilon_j\downarrow0,
\qquad
\phi_j\rightharpoonup\phi
\quad\text{weakly in }\mathcal H_U,
\qquad
\sup_j\|\nabla\phi_j\|_{L^p(\mathbb R^n)}<\infty .
\]
Then
\begin{equation}
\label{eq:averaged-affine-lsc}
\liminf_{j\to\infty}
\alpha_{n,p}
\left\langle
\mathfrak B_{\xi,\varepsilon_j}(\phi_j)
\right\rangle_{\mathbb S^{n-1}}
\ge
\frac p2
\int_{\mathbb R^n}
|\nabla U|^{p-2}
\left(
|\nabla\phi|^2
+
(p-2)
\frac{(\nabla U\cdot\nabla\phi)^2}{|\nabla U|^2}
\right)\,dx .
\end{equation}
\end{lemma}

\begin{proof}
Fix \(\xi\in\mathbb S^{n-1}\). We first establish the directional lower bound
\begin{equation}
\label{eq:directional-lsc-direct}
\liminf_{j\to\infty}
\mathfrak B_{\xi,\varepsilon_j}(\phi_j)
\ge
\frac{p(p-1)}2
\int_{\mathbb R^n}
|\partial_\xi U|^{p-2}
|\partial_\xi\phi|^2\,dx .
\end{equation}
For \(p=2\), this is nothing but weak lower semicontinuity of the
\(L^2\)-norm. We therefore restrict attention to the case \(p>2\).

Since
\(|\partial_\xi U|^{p-2}\le|\nabla U|^{p-2}\),
the weak convergence in \(\mathcal H_U\) yields
\begin{equation}
\label{eq:weak-conv-L2xiU}
        \partial_\xi\phi_j
\rightharpoonup
\partial_\xi\phi
\quad
\text{weakly in }
L^2(|\partial_\xi U|^{p-2}dx).
\end{equation}
Moreover,
\[
\sup_j\|\partial_\xi\phi_j\|_{L^p(\mathbb R^n)}
\le
\sup_j\|\nabla\phi_j\|_{L^p(\mathbb R^n)}
<\infty .
\]

Fix \(0< d<1/2\), and set
\[
E_{j, d}
:=
\left\{
|\varepsilon_j\partial_\xi\phi_j|
\le
        d |\partial_\xi U|
\right\}.
\]

On \(E_{j,d}^c\),
\begin{equation}
\label{eq:estimate-small-on-Ec}
\int_{E_{j, d}^c}
|\partial_\xi U|^{p-2}
|\partial_\xi\phi_j|^2\,dx
\le
d^{-(p-2)}
\varepsilon_j^{p-2}
\|\partial_\xi\phi_j\|_{L^p(\mathbb R^n)}^p
\to0.
\end{equation}

On \(E_{j, d}\),
\[
\frac{
|\partial_\xi U+\varepsilon_j\partial_\xi\phi_j|
-
|\partial_\xi U|
}{\varepsilon_j}
=
\operatorname{sgn}(\partial_\xi U)\,\partial_\xi\phi_j .
\]
In addition, by \eqref{eq:def-w},
\[
|w_{\xi,\varepsilon_j}|^{p-2}
\ge
(1- d)^{p-1}
|\partial_\xi U|^{p-2}
\qquad\text{on }E_{j, d}.
\]
Thus, by \eqref{eq:def-Bxi} and \eqref{eq:estimate-small-on-Ec}, we obtain
\[
\begin{aligned}
\mathfrak B_{\xi,\varepsilon_j}(\phi_j)
&\ge
\frac p2
\int_{\mathbb R^n}
|\partial_\xi U|^{p-2}
|\partial_\xi\phi_j|^2\,dx
+
\frac{p(p-2)}2
(1- d)^{p-1}
\int_{E_{j, d}}
|\partial_\xi U|^{p-2}
|\partial_\xi\phi_j|^2\,dx \\
& \ge
\left[
\frac p2
+
\frac{p(p-2)}2(1- d)^{p-1}
\right]
\int_{\mathbb R^n}
|\partial_\xi U|^{p-2}
|\partial_\xi\phi_j|^2\,dx
+o_j(1; d).
\end{aligned}
\]

Taking the liminf and invoking \eqref{eq:weak-conv-L2xiU}, we deduce
\[
\begin{aligned}
\liminf_{j\to\infty}
\mathfrak B_{\xi,\varepsilon_j}(\phi_j)
&\ge
\left[
\frac p2
+
\frac{p(p-2)}2(1- d)^{p-1}
\right]
\int_{\mathbb R^n}
|\partial_\xi U|^{p-2}
|\partial_\xi\phi|^2\,dx .
\end{aligned}
\]
Letting \( d\downarrow0\) yields \eqref{eq:directional-lsc-direct}.

Finally, average in \(\xi\), apply Fatou's lemma on \(\mathbb S^{n-1}\), and
then use \eqref{eq:angular-quadratic-identity-p}. This gives
\[
\begin{aligned}
\liminf_{j\to\infty}
\alpha_{n,p}
\left\langle
\mathfrak B_{\xi,\varepsilon_j}(\phi_j)
\right\rangle_{\mathbb S^{n-1}}
&\ge
\alpha_{n,p}
\frac{p(p-1)}2
\left\langle
\int_{\mathbb R^n}
|\partial_\xi U|^{p-2}
|\partial_\xi\phi|^2\,dx
\right\rangle_{\mathbb S^{n-1}}
\\
&=
\frac p2
\int_{\mathbb R^n}
|\nabla U|^{p-2}
\left(
|\nabla\phi|^2
+
(p-2)
\frac{(\nabla U\cdot\nabla\phi)^2}{|\nabla U|^2}
\right)\,dx .
\end{aligned}
\]
This proves \eqref{eq:averaged-affine-lsc}.
\end{proof}

\begin{lemma}
\label{lem:potential-normalized-compactness-affine}
Let \(2\le p<n\). Fix \(\kappa_1\in(0,1)\). There exists \(C>0\) such that
for every \(0<\kappa\le \kappa_1\), every \(0<\varepsilon\le1\), and every
\(\phi\in\dot W^{1,p}(\mathbb R^n)\) satisfying
\( \int_{\mathbb R^n}U^{p^*-2}\phi^2\,dx=1\),
one has
\begin{equation}
\label{eq:potential-compact}
\mathcal N_{\rm aff}^{\varepsilon,\kappa}(\phi)\le C
\quad\Longrightarrow\quad
\|\phi\|_{\mathcal Z_U}\le C .
\end{equation}
\end{lemma}

\begin{proof}
It suffices to prove the lower bound
\begin{equation}
\label{eq:lower-bdd}
\mathcal N_{\rm aff}^{\varepsilon,\kappa}(\phi)
\ge
c
\int_{\mathbb R^n}
|\nabla U|^{p-2}
\left(
|\nabla\phi|^2
+
(p-2)
\frac{(\nabla U\cdot\nabla\phi)^2}{|\nabla U|^2}
\right)\,dx
-
C
\int_{\mathbb R^n}
U^{p^*-2}\phi^2\,dx .
\end{equation}
Indeed, once \eqref{eq:lower-bdd} is available,
\eqref{eq:potential-compact} follows immediately from the normalization
\(\int_{\mathbb R^n}U^{p^*-2}\phi^2\,dx=1\).

We begin with the positive term. By \eqref{eq:def-Bxi},
\[
\mathfrak B_{\xi,\varepsilon}(\phi)
\ge
\frac p2
\int_{\mathbb R^n}
|\partial_\xi U|^{p-2}|\partial_\xi\phi|^2\,dx .
\]
Averaging in \(\xi\) and invoking \eqref{eq:angular-quadratic-identity-p}, we
obtain for every \(0<\kappa\le\kappa_1<1\),
\begin{equation}
\label{eq:B-lower-by-G}
\begin{aligned}
\alpha_{n,p}(1-\kappa)
\left\langle
\mathfrak B_{\xi,\varepsilon}(\phi)
\right\rangle_{\mathbb S^{n-1}}
&\ge c
\int_{\mathbb R^n}
|\nabla U|^{p-2}
\left(
|\nabla\phi|^2+(p-2)
\frac{(\nabla U\cdot\nabla\phi)^2}{|\nabla U|^2}
\right)\,dx .
\end{aligned}
\end{equation}

It remains to handle the negative variance term. We claim that for every
\(\delta>0\), there exists \(C_\delta>0\) such that
\begin{equation}
\label{eq:energy-variance-compact-bound}
\begin{aligned}
\alpha_{n,p}
\frac{\tau+1}{2A_0}
\operatorname{Var}_\xi(L_\xi(\phi))
&\le
\delta
\int_{\mathbb R^n}
|\nabla U|^{p-2}
\left(
|\nabla\phi|^2
+(p-2)
\frac{(\nabla U\cdot\nabla\phi)^2}{|\nabla U|^2}
\right)\,dx\\
&\quad+C_\delta
\int_{\mathbb R^n}U^{p^*-2}\phi^2\,dx .
\end{aligned}
\end{equation}
Assume, to the contrary, that \eqref{eq:energy-variance-compact-bound} fails.
Then for some \(\delta>0\), there exists a sequence
\(\phi_j\in\mathcal Z_U\) such that
\[
\begin{aligned}
\alpha_{n,p}
\frac{\tau+1}{2A_0}
\operatorname{Var}_\xi(L_\xi(\phi_j))
&>\delta
\int_{\mathbb R^n}
|\nabla U|^{p-2}
\left(
|\nabla\phi_j|^2+(p-2)
\frac{(\nabla U\cdot\nabla\phi_j)^2}{|\nabla U|^2}
\right)\,dx\\
&\quad
+j\int_{\mathbb R^n}U^{p^*-2}\phi_j^2\,dx .
\end{aligned}
\]
By homogeneity, we may normalize this sequence so that
\(\alpha_{n,p}
\frac{\tau+1}{2A_0}
\operatorname{Var}_\xi(L_\xi(\phi_j))=1\).
Consequently,
\[
\int_{\mathbb R^n}
|\nabla U|^{p-2}
\left(
|\nabla\phi_j|^2+(p-2)
\frac{(\nabla U\cdot\nabla\phi_j)^2}{|\nabla U|^2}
\right)\,dx\le
\delta^{-1},
\]
and
\begin{equation}
\label{eq:psijto0}
 \int_{\mathbb R^n}U^{p^*-2}\phi_j^2\,dx\le j^{-1}.
\end{equation}
Thus \((\phi_j)\) is bounded in \(\mathcal H_U\). By
\eqref{eq:weighted-compact-embedding}, after passing to a subsequence,
\[
\phi_j\rightharpoonup\phi
\quad\text{weakly in }\mathcal H_U,
\qquad
\phi_j\to\phi
\quad\text{strongly in }L^2(\mathbb R^n,U^{p^*-2}dx).
\]
In view of \eqref{eq:psijto0}, the limit must satisfy \(\phi=0\).

We now pass to the variance term. By the continuity statement already recorded
in \eqref{eq:cts-varL},
\(\operatorname{Var}_\xi(L_\xi(\phi_j))\to0\),
which contradicts the normalization. This proves
\eqref{eq:energy-variance-compact-bound}.

Combining \eqref{eq:B-lower-by-G} with
\eqref{eq:energy-variance-compact-bound}, and then choosing \(\delta>0\) small
depending only on \(\kappa_1\), we obtain \eqref{eq:lower-bdd}. The proof is
complete.
\end{proof}

\subsection{Proof of Proposition~\ref{prop:calibrated-affine-spectral}}

\begin{proof}[Proof of Proposition~\ref{prop:calibrated-affine-spectral}]
We argue by contradiction. By Cauchy--Schwarz,
\[
\begin{aligned}
|L_\xi(\phi)|^2
&=p^2
\left|
\int_{\mathbb R^n}
|\partial_\xi U|^{p-2}\partial_\xi U\,\partial_\xi\phi\,dx
\right|^2  \\
&\le p^2
\left(
\int_{\mathbb R^n}|\partial_\xi U|^p\,dx
\right)
\left(
\int_{\mathbb R^n}
|\partial_\xi U|^{p-2}|\partial_\xi\phi|^2\,dx
\right)\le
C
\int_{\mathbb R^n}
|\nabla U|^{p-2}
|\nabla\phi|^2 \,dx
\end{aligned}
\]
uniformly in \(\xi\). It follows that there exists \(C_*>0\) such that
\begin{equation}
\label{eq:Rp-Gp-P}
\begin{aligned}
&\int_{\mathbb R^n}
|\nabla U|^{p-2}
\left(
|\nabla\phi|^2+
(p-2)
\frac{(\nabla U\cdot\nabla\phi)^2}{|\nabla U|^2}
\right)\,dx
+
\alpha_{n,p}
\frac{\tau+1}{pA_0}
\operatorname{Var}_\xi(L_\xi(\phi))
\\
&\qquad\le C_*
\left[
\int_{\mathbb R^n}
|\nabla U|^{p-2}
|\nabla\phi|^2
\,dx
\right].
\end{aligned}
\end{equation}
Choose \(\kappa_0>0\) so small that
\( C_*\kappa_0\le \frac{c_{\rm sg}}4\),
and set
\( \eta_0=\frac{c_{\rm sg}}{4c_0}\).
Assume that \eqref{eq:calibrated-affine-spectral} fails. Then there exist
\[
0<\kappa_j\le\kappa_0,
\qquad
\varepsilon_j\downarrow0,
\qquad
\phi_j\in\dot W^{1,p}(\mathbb R^n),
\]
such that
\[
\|\nabla\phi_j\|_{L^p(\mathbb R^n)}=1,
\qquad
\phi_j\perp_{L^2(\mathbb R^n,U^{p^*-2}dx)}T_U\mathcal M_{\rm aff},
\]
but
\begin{equation}
\label{eq:false-calibrated}
\mathcal N_{\rm aff}^{\varepsilon_j,\kappa_j}(\phi_j)
<
\frac p2
\bigl((p^*-1)\Lambda+\eta_0\bigr)
\int_{\mathbb R^n}U^{p^*-2}\phi_j^2\,dx .
\end{equation}

For each \(j\),
\(\int_{\mathbb R^n}U^{p^*-2}\phi_j^2\,dx>0\),
for otherwise \(\phi_j=0\) almost everywhere, contradicting
\(\|\nabla\phi_j\|_{L^p(\mathbb R^n)}=1\). Define
\[
\lambda_j:=
\left(
\int_{\mathbb R^n}U^{p^*-2}\phi_j^2\,dx
\right)^{1/2},
\qquad
\widetilde\phi_j:=\lambda_j^{-1}\phi_j,
\qquad
\widetilde\varepsilon_j:=\varepsilon_j\lambda_j .
\]
Then
\[
\int_{\mathbb R^n}U^{p^*-2}\widetilde \phi_j^2\,dx=1,
\qquad
\widetilde \phi_j\perp_{L^2(\mathbb R^n,U^{p^*-2}dx)}T_U\mathcal M_{\rm aff}.
\]
Moreover, by H\"older's inequality and Sobolev's inequality,
\[
\lambda_j^2
=
\int_{\mathbb R^n}U^{p^*-2}\phi_j^2\,dx
\le
C\|\phi_j\|_{L^{p^*}(\mathbb R^n)}^2
\le
C\|\nabla\phi_j\|_{L^p(\mathbb R^n)}^2
=C.
\]
Hence
\(\widetilde\varepsilon_j=\varepsilon_j\lambda_j\to0\).
By the homogeneity of \(\mathcal N_{\rm aff}^{\varepsilon,\kappa}\),
\eqref{eq:false-calibrated} becomes
\begin{equation}
\label{eq:false-normalized}
\mathcal N_{\rm aff}^{\widetilde\varepsilon_j,\kappa_j}(\widetilde\phi_j)
<
\frac p2
\bigl((p^*-1)\Lambda+\eta_0\bigr).
\end{equation}
By Lemma~\ref{lem:potential-normalized-compactness-affine},
\( \|\widetilde\phi_j\|_{\mathcal Z_U}\le C\).
Thus, after passing to a subsequence,
by \eqref{eq:weighted-compact-embedding},
\[
\widetilde\phi_j\rightharpoonup\phi
\qquad\text{weakly in }\mathcal Z_U,
\qquad
\widetilde\phi_j\to\phi
\qquad\text{strongly in }
L^2(\mathbb R^n,U^{p^*-2}dx).
\]
Therefore
\[
\int_{\mathbb R^n}U^{p^*-2}\phi^2\,dx=1,
\qquad
\phi\perp_{L^2(\mathbb R^n,U^{p^*-2}dx)}T_U\mathcal M_{\rm aff}.
\]

Using Lemma~\ref{lem:averaged-affine-lsc},
Proposition~\ref{prop:affine-spectral-gap},
\eqref{eq:def-N-aff-FZ}, \eqref{eq:weighted-Poincare}, \eqref{eq:cts-varL} and \eqref{eq:Rp-Gp-P}, we obtain
\begin{equation}
\label{eq:N-estimate}
\begin{aligned}
&\liminf_{j\to\infty}
\frac2p
\mathcal N_{\rm aff}^{\widetilde\varepsilon_j,\kappa_j}(\widetilde\phi_j)\\
&\ge
(1-\kappa_0)
\int_{\mathbb R^n}
|\nabla U|^{p-2}
\left(
|\nabla\phi|^2
+
(p-2)
\frac{(\nabla U\cdot\nabla\phi)^2}{|\nabla U|^2}
\right)\,dx
\\
&\qquad\quad-
(1+\kappa_0)
\alpha_{n,p}
\frac{\tau+1}{pA_0}
\operatorname{Var}_\xi(L_\xi(\phi))
\\
&\ge
\int_{\mathbb R^n}
|\nabla U|^{p-2}
\left(
|\nabla\phi|^2
+
(p-2)
\frac{(\nabla U\cdot\nabla\phi)^2}{|\nabla U|^2}
\right)\,dx -
\alpha_{n,p}
\frac{\tau+1}{pA_0}
\operatorname{Var}_\xi(L_\xi(\phi))
\\
&\qquad\quad
-
\frac{c_{\rm sg}}4
\int_{\mathbb R^n}
|\nabla U|^{p-2}
|\nabla\phi|^2\,dx\\
&=Q_{\rm aff,p}(\phi)
+
(p^*-1)\Lambda
\int_{\mathbb R^n}U^{p^*-2}\phi^2\,dx-
\frac{c_{\rm sg}}4
\int_{\mathbb R^n}
|\nabla U|^{p-2}
|\nabla\phi|^2\,dx \\
&\ge
(p^*-1)\Lambda
\int_{\mathbb R^n}U^{p^*-2}\phi^2\,dx
+\frac{3}{4}
c_{\rm sg}
\int_{\mathbb R^n}
|\nabla U|^{p-2}
|\nabla\phi|^2\,dx\\
&\ge \left((p^*-1)\Lambda +\frac{3c_{\rm sg}}{4c_0}
\right)
\int_{\mathbb R^n}U^{p^*-2}\phi^2\,dx.
\end{aligned}
\end{equation}

Since
\( \int_{\mathbb R^n}U^{p^*-2}\phi^2\,dx=1\),
it follows from \eqref{eq:N-estimate} that
\[
\liminf_{j\to\infty}
\mathcal N_{\rm aff}^{\widetilde\varepsilon_j,\kappa_j}(\widetilde\phi_j)
\ge
\frac p2
\left((p^*-1)\Lambda+\frac{3c_{\rm sg}}{4c_0}\right).
\]
This contradicts \eqref{eq:false-normalized}, because
\(\eta_0=\frac {c_{\rm sg}}{4c_0}\). The contradiction proves
Proposition~\ref{prop:calibrated-affine-spectral}.
\end{proof}

\section{Qualitative compactness and global affine stability}
\label{sec:global-affine-stability}

We next turn to the global part of the argument. The starting point is an
affine normalization adapted to the convex body canonically associated with a
function in affine Sobolev theory. Following Zhang~\cite{Zhang99} in the
endpoint case, and Lutwak--Yang--Zhang~\cite{LYZ02} for \(1<p<n\), one
associates to \(u\) the unit ball of the directional derivative norm
\[
\xi\mapsto A_\xi(u)^{1/p}
=
\left(
\int_{\mathbb R^n}|\xi\cdot\nabla u|^p\,dx
\right)^{1/p}.
\]
We then invoke John's ellipsoid theorem
\cite[Theorem~10.12.2]{Schneider14} in order to select a
volume-preserving affine normalization.

\begin{lemma}
\label{lem:affine-normalization}
Let \(1<p<n\), and let \(0\ne u\in\dot W^{1,p}(\mathbb R^n)\). Then there
exists \(M\in SL(n)\) such that, with \(Su:=T_{M,0}u\),
\[
\|Su\|_{L^{p^*}}=\|u\|_{L^{p^*}},
\qquad
\mathcal E(Su)=\mathcal E(u),
\]
and
\[
\|\nabla Su\|_{L^p}^p
\le
C(n,p)\mathcal E(u)^p .
\]
\end{lemma}

\begin{proof}
Let
\( K_u=\{\xi\in\mathbb R^n:A_\xi(u)^{1/p}\le1\}\).
Then \(K_u\) is the origin-symmetric convex body.
By the polar formula, together with \eqref{eq:intro-affine-energy-p},
\[
\mathcal E(u)^p
=
m_{n,p}^{-1}\omega_n^{p/n}|K_u|^{-p/n}.
\]

By John's theorem \cite[Theorem~10.12.2]{Schneider14}, we may choose an
origin-centered ellipsoid \(E\) such that
\(E\subset K_u\subset \sqrt n\,E\).
Since \(E\) is centered at the origin, there exists \(L\in SL(n)\) such that
\( L(E)=B_r\), where \(|B_r|=|E|\). Moreover,
\[
|K_u|\le n^{n/2}|E|=n^{n/2}\omega_n r^n,
\qquad\text{so}\qquad
r^{-p}\le C(n)|K_u|^{-p/n}.
\]
Set
\( M=L^{-1}\), \(Su=T_{M,0}u\). By \eqref{eq:def-Axi},
\[
A_\xi(Su)=A_{M\xi}(u),
\qquad
K_{Su}=M^{-1}K_u=L(K_u).
\]
Hence \(B_r\subset K_{Su}\), and therefore
\[
A_\xi(Su)\le r^{-p}
\qquad \forall\,\xi\in\mathbb S^{n-1}.
\]
Averaging over \(\mathbb S^{n-1}\), we obtain
\[
m_{n,p}\|\nabla Su\|_{L^p}^p
=
\left\langle A_\xi(Su)\right\rangle_{\mathbb S^{n-1}}
\le r^{-p}.
\]
Consequently,
\[
\|\nabla Su\|_{L^p}^p
\le
C(n,p)|K_u|^{-p/n}
\le
C(n,p)\mathcal E(u)^p .
\]

Finally, since \(M\in SL(n)\), both the \(L^{p^*}(\mathbb R^n)\)-norm and the
affine energy are invariant under \(T_{M,0}\). This proves the lemma.
\end{proof}

The purpose of Lemma~\ref{lem:affine-normalization} is to remove the
anisotropic escape allowed by the affine invariance. Although
\(\mathcal E\) is invariant under \(SL(n)\), the Euclidean
Dirichlet energy need not be bounded along an affine orbit. The lemma chooses
a volume-preserving affine representative for which  the  gradient energy
is controlled by the affine energy.

\begin{proposition}
\label{prop:qualitative-affine-compactness}
Let \(2\le p<n\). Let \(u_k\in\dot W^{1,p}(\mathbb R^n)\) satisfy
\[
\|u_k\|_{L^{p^*}(\mathbb R^n)}=1,
\qquad
\delta_{\rm aff}(u_k)\to0.
\]
Then there exist
\(c_k\neq0\), \(A_k\in SL(n)\), \(\lambda_k>0\), \(x_k\in\mathbb R^n\)
such that
\begin{equation}
\label{eq:conv-w1p}
c_k^{-1}T_{\lambda_k A_k,x_k}u_k
\to U
\qquad
\text{strongly in }\dot W^{1,p}(\mathbb R^n).
\end{equation}
\end{proposition}

\begin{proof}
Since
\(\|u_k\|_{L^{p^*}(\mathbb R^n)}=1\),
\(\delta_{\rm aff}(u_k)\to0\),
it follows that
\begin{equation}
\label{eq:approx-E}
\mathcal E(u_k)^p\to S_{n,p}^p.
\end{equation}
Applying Lemma~\ref{lem:affine-normalization} and replacing \(u_k\) by
\(T_{B_k,0}u_k\), with \(B_k\in SL(n)\), we may assume, after relabeling, that
\[
\|u_k\|_{L^{p^*}(\mathbb R^n)}=1,
\qquad
\mathcal E(u_k)^p\to S_{n,p}^p,
\qquad
\|\nabla u_k\|_{L^p(\mathbb R^n)}\le C .
\]

We now apply the standard concentration--compactness/profile decomposition for
the critical Sobolev embedding
\[
\dot W^{1,p}(\mathbb R^n)\hookrightarrow L^{p^*}(\mathbb R^n),
\]
see, for instance,
\cite{Lions84,Lions84-1,Struwe84,Solimini95,Jaffard99}. After passing to a
subsequence, there exist profiles \(W^j\in\dot W^{1,p}(\mathbb R^n)\), scales
\(\lambda_{j,k}>0\), centers \(y_{j,k}\in\mathbb R^n\), and dislocations
\[
(g_{j,k}f)(x)
:=
\lambda_{j,k}^{-\frac{n-p}{p}}
f\left(\frac{x-y_{j,k}}{\lambda_{j,k}}\right),
\]
such that, for \(i\neq j\),
\[
\frac{\lambda_{i,k}}{\lambda_{j,k}}
+
\frac{\lambda_{j,k}}{\lambda_{i,k}}
+
\frac{|y_{i,k}-y_{j,k}|}{\lambda_{i,k}+\lambda_{j,k}}
\to\infty .
\]
Moreover, for every finite \(F\subset\mathbb N\),
\begin{equation}
\label{eq:decomp-uk}
u_k=\sum_{j\in F}g_{j,k}W^j+r_k^F,
\end{equation}
and, by the  Br\'ezis--Lieb lemma \cite{BL83},
\begin{equation}
\label{eq:compact-proof-pstar-decoupling}
\|u_k\|_{L^{p^*}(\mathbb R^n)}^{p^*}
=
\sum_{j\in F}\|W^j\|_{L^{p^*}(\mathbb R^n)}^{p^*}
+
\|r_k^F\|_{L^{p^*}(\mathbb R^n)}^{p^*}
+
o_k(1).
\end{equation}
Furthermore, for some exhaustion \(F_N\uparrow\mathbb N\),
\begin{equation}
\label{eq:critical-remainder-p}
\lim_{N\to\infty}
\limsup_{k\to\infty}
\|r_k^{F_N}\|_{L^{p^*}(\mathbb R^n)}=0.
\end{equation}
We shall also use the corresponding directional decoupling
\begin{equation}
\label{eq:compact-proof-directional-decoupling}
A_\xi(u_k)=
\sum_{j\in F}A_\xi(W^j)+
A_\xi(r_k^F)+
o_k(1),
\qquad
\text{uniformly for }\xi\in\mathbb S^{n-1}.
\end{equation}
Indeed, the usual gradient decoupling is stable under any fixed invertible
linear change of variables. Applying it to
\(M_{\xi,\eta}=P_\xi+\eta(I-P_\xi)\), where
\(P_\xi z=(z\cdot\xi)\xi\), and then letting \(\eta\downarrow0\), yields
\eqref{eq:compact-proof-directional-decoupling}. The convergence is uniform in
\(\xi\) by compactness of \(\mathbb S^{n-1}\).

Combining \eqref{eq:intro-affine-Lp}, \eqref{eq:sub-add}, \eqref{eq:approx-E},
\eqref{eq:compact-proof-pstar-decoupling}, and
\eqref{eq:compact-proof-directional-decoupling}, we obtain, for every finite
\(F\subset\mathbb N\),
\begin{equation}
\label{eq:subadd}
\begin{aligned}
S_{n,p}^p=
\lim_{k\to\infty}\mathcal E(u_k)^p
&\ge
\sum_{j\in F}\mathcal E(W^j)^p+
\liminf_{k\to\infty}\mathcal E(r_k^F)^p
\\&\ge
S_{n,p}^p\sum_{j\in F}\|W^j\|_{L^{p^*}(\mathbb R^n)}^p
+
S_{n,p}^p\liminf_{k\to\infty}\|r_k^F\|_{L^{p^*}(\mathbb R^n)}^p
\\
&=
S_{n,p}^p\sum_{j\in F}\|W^j\|_{L^{p^*}(\mathbb R^n)}^p
+S_{n,p}^p
\left(
1-\sum_{j\in F}\|W^j\|_{L^{p^*}(\mathbb R^n)}^{p^*}
\right)^{\frac{p}{p^*}}.
\end{aligned}
\end{equation}
Letting \(F\uparrow\mathbb N\), and using \eqref{eq:critical-remainder-p},
\eqref{eq:subadd}, and the fact that \(0<p/p^*<1\), we arrive at
\[
1\ge
\sum_{j\ge1}\|W^j\|_{L^{p^*}(\mathbb R^n)}^p
=
\sum_{j\ge1}
\left(\|W^j\|_{L^{p^*}(\mathbb R^n)}^{p^*}\right)^{\frac p{p^*}}
\ge
\left(
\sum_{j\ge1}\|W^j\|_{L^{p^*}(\mathbb R^n)}^{p^*}
\right)^{\frac p{p^*}}
=1.
\]
Hence equality can occur only if exactly one of the quantities
\(\|W^j\|_{L^{p^*}(\mathbb R^n)}^{p^*}\) is nonzero. Denote the corresponding profile by
\(W\), and let \(g_k\) be the associated dislocation. Then, by
\eqref{eq:decomp-uk},
\begin{equation}
\label{eq:exactly-one}
g_k^{-1}u_k
=
W+ \rho_k,
\qquad
\rho_k=g_k^{-1}r_k,
\end{equation}
where \(r_k\) denotes the corresponding remainder, and
\[
\rho_k\rightharpoonup0
\quad\text{weakly in }\dot W^{1,p}(\mathbb R^n),
\qquad
\| \rho_k\|_{L^{p^*}(\mathbb R^n)}\to0.
\]
Moreover,
\(\|W\|_{L^{p^*}(\mathbb R^n)}=1\),
\(\mathcal E(W)^p=S_{n,p}^p\).
Thus \(W\) is an extremal for the sharp affine \(p\)-Sobolev inequality.

By the equality classification for the affine \(p\)-Sobolev inequality
\cite{LYZ02},
\[
W=c_k\,T_{\lambda_k A_k,x_k}^{-1}U,
\qquad
c_k\neq0,
\qquad
A_k\in SL(n),
\qquad
\lambda_k>0,
\qquad
x_k\in\mathbb R^n,
\]
where \(c_k\) is chosen such that \(\|W\|_{L^{p^*}(\mathbb R^n)}=1\). Therefore, by
\eqref{eq:exactly-one},
\[
c_k^{-1}T_{\lambda_k A_k,x_k}g_k^{-1}u_k
=
U+\widetilde \rho_k,
\qquad
\widetilde \rho_k
=
c_k^{-1}T_{\lambda_k A_k,x_k} \rho_k .
\]
Since this fixed normalization acts continuously on both \(\dot W^{1,p}(\mathbb R^n)\) and
\(L^{p^*}(\mathbb R^n)\), we have
\[
\widetilde \rho_k\rightharpoonup0
\quad\text{weakly in }\dot W^{1,p}(\mathbb R^n),
\qquad
\|\widetilde \rho_k\|_{L^{p^*}(\mathbb R^n)}\to0.
\]
Furthermore, by \eqref{eq:approx-E}, affine invariance and homogeneity of the
affine energy yield
\begin{equation}
\label{eq:Eto}
\mathcal E(U+\widetilde \rho_k)^p
\to
\mathcal E(U)^p .
\end{equation}
By \eqref{eq:def-Lxi}, the map \(\xi\mapsto L_\xi\) is continuous from the
compact sphere \(\mathbb S^{n-1}\) into \((\dot W^{1,p}(\mathbb R^n))'\), with respect to
the dual norm. Hence its image
\(\{L_\xi:\xi\in\mathbb S^{n-1}\}\)
is compact in \((\dot W^{1,p}(\mathbb R^n))'\), and therefore
\begin{equation}
\label{eq:Lxi-rhok-uniform-zero}
\sup_{\xi\in\mathbb S^{n-1}}|L_\xi(\widetilde \rho_k)|\to0 .
\end{equation}

Since \(p\ge2\), the convexity inequality
\[
|a+b|^p
\ge
|a|^p+p|a|^{p-2}ab+c_p|b|^p
\]
with \(a=\partial_\xi U\) and \(b=\partial_\xi\widetilde \rho_k\), after
integration, gives, together with \eqref{eq:Lxi-rhok-uniform-zero},
\begin{equation}
\label{eq:Axi-ge}
A_\xi(U+\widetilde \rho_k)
\ge
A_\xi(U)+L_\xi(\widetilde \rho_k)+c_pA_\xi(\widetilde \rho_k)
=
A_\xi(U)+c_pA_\xi(\widetilde \rho_k)+o_k(1),
\end{equation}
uniformly for \(\xi\in\mathbb S^{n-1}\).

Assume by contradiction that
\(\|\nabla\widetilde \rho_k\|_{L^p(\mathbb R^n)}\not\to0\). Then, by  \eqref{eq:identify-z} yields
\begin{equation}
\label{eq:Axitobeta}
\left\langle A_\xi(\widetilde \rho_k)\right\rangle_{\mathbb S^{n-1}}
=
m_{n,p}\|\nabla\widetilde \rho_k\|_{L^p(\mathbb R^n)}^p
\to \beta>0.
\end{equation}

Let \(\mu_k\) be the probability measure on \(\mathbb S^{n-1}\) defined by
\[
\mu_k:=
\left(
\frac{\nabla\widetilde \rho_k}{|\nabla\widetilde \rho_k|}
\right)_{\#}
\left(
\frac{|\nabla\widetilde \rho_k|^p\,dx}
{\|\nabla\widetilde \rho_k\|_{L^p(\mathbb R^n)}^p}
\right).
\]
After passing to a subsequence,
\[
\mu_k\rightharpoonup\mu
\qquad\text{weakly as probability measures on }\mathbb S^{n-1}.
\]
Then
\[
A_\xi(\widetilde \rho_k)
=
\|\nabla\widetilde \rho_k\|_{L^p(\mathbb R^n)}^p
\int_{\mathbb S^{n-1}}|\theta\cdot\xi|^p\,d\mu_k(\theta),
\]
and we have the uniform convergence
\begin{equation}
\label{eq:AtoB}
\frac{A_\xi(\widetilde \rho_k)}
{\left\langle A_\xi(\widetilde \rho_k)\right\rangle_{\mathbb S^{n-1}}}
=
\frac1{m_{n,p}}
\int_{\mathbb S^{n-1}}|\theta\cdot\xi|^p\,d\mu_k(\theta)
\to
B(\xi):=
\frac1{m_{n,p}}
\int_{\mathbb S^{n-1}}|\theta\cdot\xi|^p\,d\mu(\theta).
\end{equation}
Moreover,
\[
B\ge0,
\qquad
\left\langle B\right\rangle_{\mathbb S^{n-1}}=1.
\]

Passing to the limit in \eqref{eq:Axi-ge}, and using
\eqref{eq:Axitobeta}--\eqref{eq:AtoB}, we obtain
\[
\begin{aligned}
\liminf_{k\to\infty}
\mathcal E(U+\widetilde \rho_k)^p
&=\alpha_{n,p}
\liminf_{k\to\infty}
\Phi(A_\xi(U+\widetilde \rho_k))\\
&\ge
\alpha_{n,p}\Phi(A_0+c_p\beta B)
>
\alpha_{n,p}\Phi(A_0)
=
\mathcal E(U)^p.
\end{aligned}
\]
This contradicts \eqref{eq:Eto}. Hence
\(\|\nabla\widetilde \rho_k\|_{L^p(\mathbb R^n)}\to0\).
This proves \eqref{eq:conv-w1p}, and hence the proposition.
\end{proof}

\begin{lemma}
\label{lem:affine-modulation-p}
Let \(2\le p<n\). There exist
\(\varepsilon_0>0\), \(R_0>0\), and a modulus of continuity
\(\omega\), with \(\omega(t)\to0\) as \(t\downarrow0\), such that the following
holds.

Assume that \(w\in\dot W^{1,p}(\mathbb R^n)\) satisfies
\(\|\nabla(w-U)\|_{L^p(\mathbb R^n)}\le \varepsilon_0\).
Then there exists \(V\in\mathcal M_{\rm aff}\), with
\(\|\nabla(V-U)\|_{L^p(\mathbb R^n)}\le R_0\),
such that
\begin{equation}
\label{eq:affine-modulus-p}
\|\nabla(w-V)\|_{L^p(\mathbb R^n)}
\le
\omega\!\left(
\|\nabla(w-U)\|_{L^p(\mathbb R^n)}
\right),
\end{equation}
and
\begin{equation}
\label{eq:affine-orthogonality-at-V}
\int_{\mathbb R^n}
|V|^{p^*-2}\zeta\,(w-V)\,dx
=0
\qquad
\forall \zeta\in T_V\mathcal M_{\rm aff}.
\end{equation}
\end{lemma}

\begin{proof}
For
\[
a=(\alpha,y,\lambda,B)
\in
\mathbb R\times\mathbb R^n\times\mathbb R\times\operatorname{Sym}_0(n)
\]
near \(0\), set
\[
\Xi(a)=
(1+\alpha)\,
T_{(1+\lambda)(I+B),y}^{-1}U .
\]
We suppress rotational parameters, since the radial profile \(U\) is invariant
under rotations. Thus \(\Xi\) provides a local parametrization of
\(\mathcal M_{\rm aff}\) near \(U\), modulo the redundant rotational
directions.

For \(w\) sufficiently close to \(U\), consider the minimization problem
\[
\mathscr F_w(a)
:=
\frac{1}{p^*}\int_{\mathbb R^n}|\Xi(a)|^{p^*}\,dx
-
\frac1{p^*-1}
\int_{\mathbb R^n}
|\Xi(a)|^{p^*-2}\Xi(a)\,w\,dx
\]
over a sufficiently small closed ball in parameter space. Exactly as in
\cite[Lemma~4.1]{FZ22}, once
\(\|\nabla(w-U)\|_{L^p(\mathbb R^n)}\) is sufficiently small, the minimizer is
an interior point and depends continuously on \(w\). Writing this minimizer as
\(a_w\), and setting \(V=\Xi(a_w)\), we obtain \eqref{eq:affine-modulus-p}.
Since \(a_w\) is an interior minimizer, the first variation of
\(\mathscr F_w\) vanishes along every tangent direction. It follows that, for
every \(\zeta\in T_V\mathcal M_{\rm aff}\), one has
\eqref{eq:affine-orthogonality-at-V}. This proves the lemma.
\end{proof}

\subsection{Proof of Theorem~\ref{thm:intro-main-affine-p}}

\begin{proof}[Proof of Theorem~\ref{thm:intro-main-affine-p}]
We begin with the local bound. More precisely, we claim that there exist
\(\delta_0>0\) and \(c_0>0\), depending only on \(n,p\), such that
\begin{equation}
\label{eq:local-affine-p-stability}
        d_{\rm aff}(u,\mathcal M_{\rm aff}) \le\delta_0
\quad\Longrightarrow\quad
\widehat \delta_{\rm aff}(u)\ge c_0  d_{\rm aff}(u,\mathcal M_{\rm aff})^p,
\end{equation}
where
\begin{equation}
\label{eq:affine-distance-p}
d_{\rm aff}(u,\mathcal M_{\rm aff})
=
\inf_{\substack{a\in\mathbb R\\ A\in SL(n),\,\lambda>0,\,
x_0\in\mathbb R^n}}
\frac{
\left\|
\nabla\left(T_{\lambda A,x_0}u-aU\right)
\right\|_{L^p}
}{
\left\|
\nabla\left(T_{\lambda A,x_0}u\right)
\right\|_{L^p}
}.
\end{equation}
Here
\[
\widehat  \delta_{\rm aff}(u)
=
\frac{\mathcal E(u)}
{S_{n,p}\|u\|_{L^{p^*}(\mathbb R^n)}}-1,
\qquad
u\in\dot W^{1,p}(\mathbb R^n)\backslash\{0\}.
\]

Set \( d= d_{\rm aff}(u,\mathcal M_{\rm aff})\). The case \( d=0\) is trivial, so we may assume
\(0< d\le\delta_0\). Choose admissible parameters
\(a_0\in\mathbb R\), \(A_0\in SL(n)\), \(\lambda_0>0\), and
\(x_0\in\mathbb R^n\), and write
\( T_0=T_{\lambda_0A_0,x_0}\),
so that
\begin{equation}
\label{eq:near-minimizer-relative-distance}
\|\nabla(T_0u-a_0U)\|_{L^p(\mathbb R^n)}
\le
2 d\,\|\nabla(T_0u)\|_{L^p(\mathbb R^n)}.
\end{equation}
If \(\delta_0<1/4\), then necessarily \(a_0\ne0\). Set \(w=a_0^{-1}T_0u\).
Then \eqref{eq:near-minimizer-relative-distance} yields
\[
\|\nabla(w-U)\|_{L^p(\mathbb R^n)}
\le
\frac{2 d}{1-2 d}\|\nabla U\|_{L^p(\mathbb R^n)}.
\]
After decreasing \(\delta_0\) if necessary,
Lemma~\ref{lem:affine-modulation-p} applies to \(w\). Thus there exists
\(V\in\mathcal M_{\rm aff}\), close to \(U\),
such that \eqref{eq:affine-modulus-p} and \eqref{eq:affine-orthogonality-at-V} hlod.
Write \(V=b\,T_{\lambda A,x_1}^{-1}U\), \(b\ne0\),
with \(\lambda A\) close to \(I\), and define \(\phi\) by
\[
b^{-1}T_{\lambda A,x_1}w=U+\phi.
\]
After changing variables in the orthogonality condition \eqref{eq:affine-orthogonality-at-V}, we obtain
\[
\phi\perp_{L^2(\mathbb R^n,U^{p^*-2}dx)}T_U\mathcal M_{\rm aff}.
\]
Moreover, by \eqref{eq:affine-modulus-p},
\[
\|\nabla\phi\|_{L^p(\mathbb R^n)}
\le
C\|\nabla(w-V)\|_{L^p(\mathbb R^n)}
\le
C\omega(\|\nabla(w-U)\|_{L^p(\mathbb R^n)}).
\]
Hence, after possibly decreasing \(\delta_0\), the quantity
\(\|\nabla\phi\|_{L^p(\mathbb R^n)}\) lies below all the smallness thresholds
needed below.

If \(\phi=0\), then the conclusion is immediate. Assume therefore that
\(\phi\ne0\), and set
\[
\varepsilon:=\|\nabla\phi\|_{L^p(\mathbb R^n)},
\qquad
\phi:=\varepsilon^{-1}\phi .
\]
Then
\[
\|\nabla\phi\|_{L^p(\mathbb R^n)}=1,
\qquad
\phi\perp_{L^2(\mathbb R^n,U^{p^*-2}dx)}T_U\mathcal M_{\rm aff}.
\]

Fix \(0<\kappa\le\kappa_0\), where \(\kappa_0\) is given by
Proposition~\ref{prop:calibrated-affine-spectral}.
By \eqref{eq:D-ge-2}, affine invariance and homogeneity of the deficit,
\begin{equation}
\label{eq:delta-lower-by-psi}
\begin{aligned}
\widehat  \delta_{\rm aff}(u)
=
\widehat  \delta_{\rm aff}(U+\phi)
&\ge
c
\left[
\frac{\mathcal E(U+\phi)^p}
{S_{n,p}^p\|U+\phi\|_{L^{p^*}(\mathbb R^n)}^p}
-1
\right]\\
& \ge
c\,\delta_{\rm aff}(U+\phi)
=c\delta_{\rm aff}(U+\varepsilon\phi)
\ge
c\varepsilon^p  =
c\|\nabla\phi\|_{L^p(\mathbb R^n)}^p.
\end{aligned}
\end{equation}

Finally, since
\[
w=a_0^{-1}T_0u,
\qquad
b^{-1}T_{\lambda A,x_1}w=U+\phi,
\]
we have
\[
\begin{aligned}
d&\le
\frac{
\|\nabla(T_{\lambda A,x_1}T_0u-a_0bU)\|_{L^p(\mathbb R^n)}
}{
\|\nabla(T_{\lambda A,x_1}T_0u)\|_{L^p(\mathbb R^n)}
}
=
\frac{
\|\nabla\phi\|_{L^p(\mathbb R^n)}
}{
\|\nabla(U+\phi)\|_{L^p(\mathbb R^n)}
}
\le
C\|\nabla\phi\|_{L^p(\mathbb R^n)}.
\end{aligned}
\]
Combining this with \eqref{eq:delta-lower-by-psi} proves
\eqref{eq:local-affine-p-stability}.

We now turn to the global estimate. Assume, for contradiction, that
\eqref{eq:intro-main-affine-p} fails. Then there exists a sequence
\(u_k\ne0\) such that
\begin{equation}
\label{eq:false-global}
\widehat \delta_{\rm aff}(u_k)
<
\frac1k\,  d_{\rm aff}(u_k,\mathcal M_{\rm aff})^p.
\end{equation}
Since \( d_{\rm aff}(u_k,\mathcal M_{\rm aff})\le1\), it follows that
\(  \widehat \delta_{\rm aff}(u_k)\to0\).

Both \(  \widehat\delta_{\rm aff}\) and \( d\) are invariant under multiplication by
nonzero scalars. Hence we may normalize so that
\(\|u_k\|_{L^{p^*}(\mathbb R^n)}=1\).
By Proposition~\ref{prop:qualitative-affine-compactness}, after passing to a
subsequence, there exist admissible affine transformations and nonzero scalars
such that
\[
c_k^{-1}T_{\lambda_kM_k,x_k}u_k
\to U
\qquad
\text{strongly in }\dot W^{1,p}(\mathbb R^n).
\]
In particular,
\( d_{\rm aff}(u_k,\mathcal M_{\rm aff})\to0\).
Hence, for all sufficiently large \(k\), the local estimate
\eqref{eq:local-affine-p-stability} yields
\[
\widehat \delta_{\rm aff}(u_k)
\ge
c_0 d_{\rm aff}(u_k,\mathcal M_{\rm aff})^p,
\]
which contradicts \eqref{eq:false-global}. This proves
\eqref{eq:intro-main-affine-p}.

It remains to show that the exponent \(p\) is optimal. Let
\(0\ne\zeta\in C_c^\infty(\mathbb R^n)\), and set
\[
\zeta_R(x):=\zeta(x-Re_1),
\qquad
u_{\varepsilon,R}:=U+\varepsilon\zeta_R .
\]
As \(R\to\infty\), the support of \(\zeta_R\) escapes to infinity relative to
the bubble \(U\). Hence, by the Br\'ezis--Lieb lemma~\cite{BL83},
\[
\begin{aligned}
&\|\nabla u_{\varepsilon,R}\|_{L^p(\mathbb R^n)}^p
=
\|\nabla U\|_{L^p(\mathbb R^n)}^p+
\varepsilon^p
\|\nabla\zeta\|_{L^p(\mathbb R^n)}^p
+o_R(1).\\
&\|u_{\varepsilon,R}\|_{L^{p^*}(\mathbb R^n)}^p
=
\|U\|_{L^{p^*}(\mathbb R^n)}^p+
O(\varepsilon^{p^*})+
o_R(1).
\end{aligned}
\]
Thus we obtain
\begin{equation}
\label{eq:Sob-deficit-sharpness}
\begin{aligned}
c  \widehat \delta_{\rm aff, p}(u_{\varepsilon,R}) \le \delta_{\rm aff, p}(u_{\varepsilon,R})
&\le
\|\nabla u_{\varepsilon,R}\|_{L^p(\mathbb R^n)}^p-
S_{n,p}^p\|u_{\varepsilon,R}\|_{L^{p^*}(\mathbb R^n)}^p\\
&\le
\varepsilon^p
\|\nabla\zeta\|_{L^p(\mathbb R^n)}^p
-S_{n,p}^p O(\varepsilon^{p^*})
+o_R(1)                                      \\
&\le
C\varepsilon^p+o_R(1).
\end{aligned}
\end{equation}
It remains to compare the affine distance with \(\varepsilon\). By
\eqref{eq:affine-distance-p}, for each fixed \(R\),
\begin{align*}
d_{\rm aff}(U+\varepsilon\zeta_R,\mathcal M_{\rm aff})
&=
\frac{\varepsilon}{\|\nabla U\|_{L^p(\mathbb R^n)}}
\operatorname{dist}_{\dot W^{1,p}}
\bigl(\zeta_R,T_U\mathcal \mathcal M_{\rm aff}\bigr)
+o_R(\varepsilon),\\
&=
\frac{\varepsilon}{\|\nabla U\|_{L^p(\mathbb R^n)}}
\bigl(\|\nabla\zeta\|_{L^p(\mathbb R^n)}+o_R(1)\bigr)
+o_R(\varepsilon),
\qquad \varepsilon\downarrow0 .
\end{align*}
Here the distance on the right-hand side is taken with respect to the
\(\dot W^{1,p}(\mathbb R^n)\)-seminorm.

After normalizing \(\|\nabla\zeta\|_{L^p(\mathbb R^n)}=1\), choosing \(R\) large
and then \(\varepsilon>0\) sufficiently small, we obtain
\begin{equation}
\label{eq:sharpness-distance-lower}
d_{\rm aff}(U+\varepsilon\zeta_R,\mathcal M_{\rm aff})
\ge c\,\varepsilon .
\end{equation}

On the other hand, by \eqref{eq:Sob-deficit-sharpness}, taking \(R\) large and
then letting \(\varepsilon\downarrow0\), and using
\eqref{eq:sharpness-distance-lower}, we obtain
\[
\widehat  \delta_{\rm aff}(U+\varepsilon\zeta_R)
\le
C\,d_{\rm aff}(U+\varepsilon\zeta_R,\mathcal M_{\rm aff})^p .
\]
Thus the exponent \(p\) is optimal: no estimate with exponent
\(\alpha<p\) can hold uniformly.
\end{proof}

\section{Spectral decomposition of the affine \(p\)-Sobolev Hessian}
\label{sec:sector-decomposition-affine-p}

The purpose of this section is to identify the nullspace of the affine
Hessian \(Q_{\rm aff,p}\). We shall prove that
\[
T_U\mathcal \mathcal M_{\rm aff}
=
\ker Q_{\rm aff,p}
=
\operatorname{span}
\left\{
U,\ Z_0,\ \partial_{x_1}U,\dots,\partial_{x_n}U,\ x\cdot B\nabla U
\right\},
\]
where
\[
B=B^T,\ \operatorname{tr}B=0,
\qquad
Z_0=\frac{n-p}{p}U+x\cdot\nabla U .
\]
The radial and translation sectors rely on the classical nondegeneracy theory
for the critical \(p\)-Sobolev bubble \cite{PV21,FN19,FZ22}. The trace-free
degree-two affine modes are produced by the variance correction, and the
remaining sectors are handled by a Funk--Hecke computation, in the spirit of
the affine fractional Hilbertian analysis in \cite{FLZ26-1}.

Throughout this section, the kernels of \(Q_{\rm Sob,p}\) and
\(Q_{\rm aff,p}\) are taken in the space \(\mathcal Z_U\).

\subsection{Physical angular sectors}\label{subsec:phy}

Let \(\phi\in\mathcal Z_U\). Since \(U=U(r)\) is radial, the spherical
harmonic decomposition is compatible with the quadratic structure induced by
\(\mathcal Z_U\).

For a.e. \(r>0\), the map
\(\theta\mapsto \phi(r,\theta)\) belongs to \(L^2(\mathbb S^{n-1})\). Since
the spherical harmonics form a complete orthogonal basis of
\(L^2(\mathbb S^{n-1})\), we may write
\begin{equation}
\label{eq:decom-phi}
\phi(r,\theta)=
\sum_{\ell=0}^{\infty}
\sum_{m=1}^{d_\ell}
f_{\ell,m}(r)Y_{\ell,m}(\theta),
\qquad x=r\theta,
\end{equation}
where the expansion is understood in the angular
\(L^2(\mathbb S^{n-1})\)-sense for a.e. \(r>0\), and
\[
f_{\ell,m}(r)=
\left\langle
\phi(r,\cdot),Y_{\ell,m}
\right\rangle_{\mathbb S^{n-1}}=
\frac1{|\mathbb S^{n-1}|}
\int_{\mathbb S^{n-1}}
\phi(r,\theta)Y_{\ell,m}(\theta)\,d\sigma(\theta) .
\]
Here
\begin{equation}
\label{eq:sphere-harmonics}
-\Delta_{\mathbb S^{n-1}}Y_{\ell,m}
=
\lambda_\ell Y_{\ell,m},
\qquad
\lambda_\ell=\ell(\ell+n-2),
\end{equation}
and
\begin{equation}
\label{eq:normalization-harmonics}
\left\langle Y_{\ell,m}Y_{\ell,m'}\right\rangle_{\mathbb S^{n-1}}
=\delta_{mm'}.
\end{equation}

In polar coordinates, in the weak sense,
\begin{equation}
\label{eq:nabla-phi}
\begin{aligned}
\nabla\phi
& =
\theta\,\partial_r\phi+
\frac1r\,\nabla_{\mathbb S^{n-1}}\phi\\
&= \theta\,\sum_{\ell,m}f'_{\ell,m}(r)Y_{\ell,m}(\theta)+\frac1r\,\sum_{\ell,m}f_{\ell,m}(r)
\nabla_{\mathbb S^{n-1}}Y_{\ell,m}(\theta).
\end{aligned}
\end{equation}
Therefore, by \eqref{eq:sphere-harmonics}, \eqref{eq:normalization-harmonics}
and \eqref{eq:nabla-phi},
\begin{align*}
\int_{\mathbb R^n}
&|U'|^{p-2}
\left(
|\nabla\phi|^2+
(p-2)(\theta\cdot\nabla\phi)^2
\right)\,dx\\
&\qquad=
\int_0^\infty
|U'|^{p-2}
\int_{\mathbb S^{n-1}}
\left[
(p-1)|\partial_r\phi|^2+
\frac1{r^2}
|\nabla_{\mathbb S^{n-1}}\phi|^2
\right]
d\sigma\, r^{n-1}\,dr\\
&\qquad=|\mathbb S^{n-1}|
\sum_{\ell=0}^{\infty}\sum_{m=1}^{d_\ell}
\int_0^\infty
|U'|^{p-2}
\left[
(p-1)|f'_{\ell,m}|^2
+
\lambda_\ell\frac{f_{\ell,m}^2}{r^2}
\right]
r^{n-1}\,dr .
\end{align*}
Similarly,
\[
\int_{\mathbb R^n}
U^{p^*-2}\phi^2\,dx
=
|\mathbb S^{n-1}|
\sum_{\ell=0}^{\infty}\sum_{m=1}^{d_\ell}
\int_0^\infty
U^{p^*-2}f_{\ell,m}^2 r^{n-1}\,dr .
\]
Finally, since every nonconstant spherical harmonic has zero mean,
\[
\int_{\mathbb R^n}U^{p^*-1}\phi\,dx
=
|\mathbb S^{n-1}|
\int_0^\infty
U^{p^*-1}f_{0,1}(r)r^{n-1}\,dr .
\]
Accordingly, by \eqref{eq:classical-p-homogeneous-Hessian},
\[
Q_{\rm Sob,p}(\phi)
=
Q_{\rm Sob,p}^{(0)}(f_{0,1})
+
\sum_{\ell=1}^{\infty}\sum_{m=1}^{d_\ell}
Q_{\rm Sob,p}^{(\ell)}(f_{\ell,m}),
\]
where, for \(\ell\ge1\),
\begin{equation}
\label{eq:classical-p-sector}
\begin{aligned}
Q_{\rm Sob,p}^{(\ell)}(f)
&:=
|\mathbb S^{n-1}|
\int_0^\infty |U'|^{p-2}
\left[
(p-1)|f'|^2+
\lambda_\ell\frac{f^2}{r^2}
\right]r^{n-1}\,dr
\\
&\quad-
|\mathbb S^{n-1}|
(p^*-1)\Lambda
\int_0^\infty U^{p^*-2}f^2r^{n-1}\,dr ,
\end{aligned}
\end{equation}
and in the radial sector
\begin{equation}
\label{eq:classical-p-radial-sector}
\begin{aligned}
Q_{\rm Sob,p}^{(0)}(f)
&:=
|\mathbb S^{n-1}|
\int_0^\infty |U'|^{p-2}(p-1)|f'|^2 r^{n-1}\,dr\\
&\,-
|\mathbb S^{n-1}|
(p^*-1)\Lambda
\int_0^\infty U^{p^*-2}f^2 r^{n-1}\,dr
\\&\,+
(p^*-p)\Lambda \left(  |\mathbb S^{n-1}|
\int_0^\infty U^{p^*}r^{n-1}\,dr\right)^{-1}
\left(
|\mathbb S^{n-1}|
\int_0^\infty U^{p^*-1}f r^{n-1}\,dr
\right)^2 .
\end{aligned}
\end{equation}

\subsection{Classical nondegeneracy}
\label{subsec:classical-nondegeneracy}

We next recall the nondegeneracy result for the critical \(p\)-Sobolev
bubble, due to Pistoia--Vaira~\cite{PV21}; see also
Figalli--Neumayer~\cite[Proposition~3.1]{FN19} and
Figalli--Zhang~\cite{FZ22}. For the homogeneous Hessian
\eqref{eq:classical-p-homogeneous-Hessian}, one has
\begin{equation*}
\ker Q_{\rm Sob,p}=
\operatorname{span}
\left\{
U,\ Z_0,\ \partial_{x_1}U,\dots,\partial_{x_n}U
\right\}.
\end{equation*}
Here \(U\) is the amplitude mode, \(Z_0\) is the dilation mode, and the
functions \(\partial_{x_j}U\) are the translation modes.

Since \(U\) is radial, the angular projections preserve \(\mathcal Z_U\), and
the decomposition from Subsection~\ref{subsec:phy} separates the physical
angular sectors. The modes
\(U\), \(Z_0\)
belong to the radial sector. Moreover,
\[
\partial_{x_j}U(x)=U'(r)\theta_j,
\qquad j=1,\dots,n,
\]
and \(\theta_1,\dots,\theta_n\) span the first spherical harmonic sector.
Therefore, for radial components,
\begin{equation}
\label{eq:redial-kernel}
Q_{\rm Sob,p}[\phi_0]=0
\quad\Longleftrightarrow\quad
\phi_0\in\operatorname{span}\{U,Z_0\},
\end{equation}
and in the first angular sector,
\begin{equation}
\label{eq:1st-kernel}
Q_{\rm Sob,p}[\phi_1]=0
\quad\Longleftrightarrow\quad
\phi_1\in
\operatorname{span}
\left\{
\partial_{x_1}U,\dots,\partial_{x_n}U
\right\}.
\end{equation}
Here
\[
\phi_\ell(r,\theta)
=
\sum_{m=1}^{d_\ell}
f_{\ell,m}(r)Y_{\ell,m}(\theta),
\qquad
Q_{\rm Sob,p}[\phi_\ell]
:=
\sum_{m=1}^{d_\ell}
Q_{\rm Sob,p}^{(\ell)}(f_{\ell,m}).
\]
For every \(\ell\ge2\), the classical \(p\)-Sobolev Hessian has trivial kernel
in the \(\ell\)-th physical sector.

\subsection{Funk--Hecke analysis of the  affine correction \(\mathcal R_p(\phi)\)}

We now turn to the genuinely affine part of the Hessian. Since the correction \(\mathcal R_p\)
is defined through directional averages, its action on the physical angular
decomposition is encoded by the corresponding Funk--Hecke coefficients.

By \eqref{eq:nabla-phi},
\begin{equation}
\label{eq:xiphi}
\begin{aligned}
\partial_\xi\phi&=
(\theta\cdot\xi)\partial_r\phi
+
\frac1r
\xi\cdot\nabla_{\mathbb S^{n-1}}\phi\\&=
\sum_{\ell=0}^{\infty}
\sum_{m=1}^{d_\ell}
\left[
f'_{\ell,m}(r)Y_{\ell,m}(\theta)(\theta\cdot\xi)+
\frac{f_{\ell,m}(r)}{r}
\xi\cdot\nabla_{\mathbb S^{n-1}}Y_{\ell,m}(\theta)
\right].
\end{aligned}
\end{equation}

Substituting \eqref{eq:xiU}, \eqref{eq:decom-phi} and \eqref{eq:xiphi} into
\eqref{eq:def-Lxi}, we obtain
\begin{equation}
\label{eq:Lxi-local-full-before-FH}
\begin{aligned}
L_\xi(\phi)
&= p\int_{\mathbb R^n}
|\partial_\xi U|^{p-2}\partial_\xi U\,\partial_\xi\phi\,dx \\
&=p\sum_{\ell=0}^{\infty}
\sum_{m=1}^{d_\ell}
\int_0^\infty
|U'|^{p-2}U' r^{n-1}
\Bigg[
f'_{\ell,m}
\int_{\mathbb S^{n-1}}
|\theta\cdot\xi|^pY_{\ell,m}(\theta)\,d\sigma(\theta)
\\
&\qquad\qquad\qquad\qquad
+
\frac{f_{\ell,m}}{r}
\int_{\mathbb S^{n-1}}
|\theta\cdot\xi|^{p-2}(\theta\cdot\xi)
\,\xi\cdot\nabla_{\mathbb S^{n-1}}Y_{\ell,m}(\theta)\,d\sigma(\theta)
\Bigg]dr .
\end{aligned}
\end{equation}
We next compute the two angular integrals. By the Funk--Hecke identity
\cite[Theorem~1.2.9]{DX13},
\begin{equation}
\label{eq:p-FH-identity}
\frac1{|\mathbb S^{n-1}|}
\int_{\mathbb S^{n-1}}
|\theta\cdot\xi|^pY_{\ell,m}(\theta)\,d\sigma(\theta)
=
d_{\ell,p}Y_{\ell,m}(\xi).
\end{equation}

The second angular term is reduced to the same coefficient by integration by
parts on the sphere. Indeed,
\[
\nabla_{\mathbb S^{n-1}}|\theta\cdot\xi|^p
=
p|\theta\cdot\xi|^{p-2}(\theta\cdot\xi)
\bigl(\xi-(\theta\cdot\xi)\theta\bigr),
\]
and therefore, by \eqref{eq:sphere-harmonics} and
\eqref{eq:normalization-harmonics},
\begin{equation}
\label{eq:identify-angular}
\int_{\mathbb S^{n-1}}
|\theta\cdot\xi|^{p-2}(\theta\cdot\xi)
\,\xi\cdot\nabla_{\mathbb S^{n-1}}Y_{\ell,m}\,d\sigma
=
|\mathbb S^{n-1}|\frac{\lambda_\ell}{p} d_{\ell,p}Y_{\ell,m}(\xi).
\end{equation}

Substituting \eqref{eq:p-FH-identity} and \eqref{eq:identify-angular} into
\eqref{eq:Lxi-local-full-before-FH} yields
\begin{equation*}
L_\xi(\phi)
=p
\sum_{\ell=0}^{\infty}
\sum_{m=1}^{d_\ell}
        d_{\ell,p}Y_{\ell,m}(\xi)R_\ell(f_{\ell,m}),
\end{equation*}
where
\begin{equation}
\label{eq:Rell-local-p}
R_\ell(f)
:=
|\mathbb S^{n-1}|
\int_0^\infty
|U'|^{p-2}U'
\left(f'+\frac{\lambda_\ell}{p}\frac{f}{r}\right)
r^{n-1}\,dr .
\end{equation}
In particular, for a fixed component \(f(r)Y_{\ell,m}(\theta)\),
\begin{equation*}
L_\xi(fY_{\ell,m})
=
p d_{\ell,p}Y_{\ell,m}(\xi)R_\ell(f).
\end{equation*}

The radial contribution \(\ell=0\) is independent of \(\xi\), since
\(Y_{0,1}\equiv1\). Hence it contributes to
\(\langle L_\xi(\phi)\rangle_{\mathbb S^{n-1}}\), but not to the variance.
Using \eqref{eq:normalization-harmonics},
it follows from \eqref{eq:def-varxi} that
\begin{equation}
\label{eq:variance-Lxi-sector-block}
\operatorname{Var}_{\xi\in\mathbb S^{n-1}}
\bigl(L_\xi(\phi)\bigr)
=
p^2
\sum_{\ell=1}^{\infty}
\sum_{m=1}^{d_\ell}
        d_{\ell,p}^2R_\ell(f_{\ell,m})^2 .
\end{equation}

It remains to compute the coefficients \( d_{\ell,p}\).

\medskip
Set \(\mu=\frac{n-2}{2}>0\).
By \cite[Theorem~1.2.9]{DX13}, the Funk--Hecke coefficient in
\eqref{eq:p-FH-identity} is
\begin{equation*}
        d_{\ell,p}
=
\frac{|\mathbb S^{n-2}|}
{|\mathbb S^{n-1}|\,C_\ell^\mu(1)}
\int_{-1}^{1}
|t|^p C_\ell^\mu(t)(1-t^2)^{\mu-\frac12}\,dt .
\end{equation*}
By \cite[Section~10.9~(16)]{EMOT81}, namely
\(C_\ell^\mu(-t)=(-1)^\ell C_\ell^\mu(t)\),
we immediately obtain
\begin{equation}
\label{eq:rho-odd-zero-p}
        d_{\ell,p}=0
\qquad
\text{for every odd }\ell,
\end{equation}

We now compute the even coefficients. Let \(\ell=2m\). By
\cite[Section~10.9~(21)]{EMOT81},
\[
\begin{aligned}
C_{2m}^{\mu}(\sqrt u)
&=
(-1)^m\frac{(\mu)_m}{m!}
{}_2F_1
\left(
-m,\ m+\mu;\ \frac12;\ u
\right)  \\
&=
(-1)^m\frac{(\mu)_m}{m!}
\sum_{k=0}^{m}
\frac{(-m)_k(m+\mu)_k}{\left(\frac12\right)_k}
\frac{u^k}{k!}.
\end{aligned}
\]
Here \((a)_0:=1\) and \((a)_k:=a(a+1)\cdots(a+k-1)\), \(k\ge1\),
denotes the Pochhammer symbol.
Then
\begin{align*}
J_{2m,p}&:=\int_{-1}^{1}|t|^p C_{2m}^{\mu}(t)(1-t^2)^{\mu-\frac12}\,dt \\
&=\int_0^1
u^{\frac{p-1}{2}}
(1-u)^{\mu-\frac12}
C_{2m}^{\mu}(\sqrt u)\,du\\
&=
(-1)^m\frac{(\mu)_m}{m!}
\int_0^1
u^{\frac{p-1}{2}}(1-u)^{\mu-\frac12}
{}_2F_1
\left(
-m,\ m+\mu;\ \frac12;\ u
\right)\,du\\
&=
(-1)^m\frac{(\mu)_m}{m!}
B(\frac{p+1}{2},\mu+\frac12)\,
{}_3F_2
\left(
\begin{matrix}
-m,\ m+\mu,\ \frac{p+1}{2}\\
\frac12,\ \mu+\frac{p}{2}+1
\end{matrix}
;1
\right)\\
&=(-1)^m
\frac{(\mu)_m}{m!}
B(\frac{p+1}{2},\mu+\frac12)\,
\frac{
\left(\mu+\frac12\right)_m
\left(-\frac p2\right)_m
}{
\left(\frac12\right)_m
\left(\mu+\frac p2+1\right)_m
}.
\end{align*}
In the last step we use Saalsch\"utz's summation formula
\cite[Theorem~2.2.6]{AAR99}. In the present parameters,
\[
\begin{aligned}
&{}_3F_2
\left(
\begin{matrix}
-m,\ m+\mu,\ \frac{p+1}{2}\\
\frac12,\ \mu+\frac p2+1
\end{matrix}
;1
\right)
=
\sum_{k=0}^{m}
\frac{
(-m)_k(m+\mu)_k\left(\frac{p+1}{2}\right)_k
}{
\left(\frac12\right)_k
\left(\mu+\frac p2+1\right)_k
}
\frac1{k!}
=
\frac{
\left(\mu+\frac12\right)_m
\left(-\frac p2\right)_m
}{
\left(\frac12\right)_m
\left(\mu+\frac p2+1\right)_m
}.
\end{aligned}
\]

On the other hand, by \cite[Corollary~2.2.3]{AAR99},
\[
J_{0,p}=B(\frac{p+1}{2},\mu+\frac12),
\qquad
C_{2m}^{\mu}(1)
=
\frac{(\frac{1}{2}-m-\mu)_{m}}{\left(\frac{1}{2}\right)_m}
=
\frac{
(\mu)_m\left(\mu+\frac12\right)_m
}{
m!\left(\frac12\right)_m
}.
\]
Thus we obtain, for each \( m\ge0\),
\begin{equation*}
\frac{ d_{2m,p}}{ d_{0,p}}
=
\frac{J_{2m,p}}{C_{2m}^{\mu}(1)J_{0,p}}
=
(-1)^m
\frac{\left(-\frac p2\right)_m}
{\left(\mu+\frac p2+1\right)_m}= (-1)^m
\frac{\left(-\frac p2\right)_m}
{\left(\frac n2+\frac p2\right)_m}.
\end{equation*}
Consequently, for \(m\ge1\),
\begin{equation}
\label{eq:rho-ratio-p}
\left|
\frac{ d_{2m,p}}{ d_{2,p}}
\right|=
\prod_{j=1}^{m-1}
\frac{\left|j-\frac p2\right|}
{j+\frac n2+\frac p2}.
\end{equation}

\subsection{Affine correction in physical sectors}

By \eqref{eq:R-p-variance}, \eqref{eq:decom-phi} and
\eqref{eq:variance-Lxi-sector-block},
\begin{equation*}
\mathcal R_p(\phi)
=
\alpha_{n,p}
\frac{\tau+1}{pA_0}
\operatorname{Var}_{\xi\in\mathbb S^{n-1}}
\bigl(L_\xi(\phi)\bigr)= \sum_{\ell=1}^{\infty}
\mathcal R_p[\phi_\ell],
\end{equation*}
where
\begin{equation*}
\mathcal R_p[\phi_\ell]
=
\sum_{m=1}^{d_\ell}
\mathcal R_p^{(\ell)}(f_{\ell,m})=  \beta_{\ell,p}
\sum_{m=1}^{d_\ell}
R_\ell(f_{\ell,m})^2,
\qquad
\beta_{\ell,p}= \alpha_{n,p}
\frac{\tau+1}{pA_0}
p^2 d_{\ell,p}^2 .
\end{equation*}
Consequently, by \eqref{eq:Q-aff-decomp-p},
\begin{equation}
\label{eq:Q-aff-full-sector-decomposition}
Q_{\rm aff,p}(\phi)
=
Q_{\rm Sob,p}(\phi)
-
\sum_{\ell=1}^{\infty}
\mathcal R_p[\phi_\ell].
\end{equation}
Moreover, by \eqref{eq:rho-odd-zero-p},
\(\beta_{\ell,p}=0\) for every odd \(\ell\).
For even degrees \(\ell=2m\ge2\), since \(\beta_{\ell,p}\) is proportional to
\( d_{\ell,p}^2\), \eqref{eq:rho-ratio-p} gives
\begin{equation}
\label{eq:beta-ratio-p}
\frac{\beta_{2m,p}}{\beta_{2,p}}
=
\left|
\frac{ d_{2m,p}}{ d_{2,p}}
\right|^2=
\prod_{j=1}^{m-1}
\left(
\frac{\left|j-\frac p2\right|}
{j+\frac n2+\frac p2}
\right)^2 .
\end{equation}

Finally, if \(\ell=0\), then \(L_\xi\) is independent of \(\xi\), and hence
\(\operatorname{Var}_{\xi\in\mathbb S^{n-1}}\bigl(L_\xi\bigr)=0\).
Accordingly, the affine correction vanishes in the radial sector:
\begin{equation}
\label{eq:radial-affine-correction-vanishes}
Q_{\rm aff,p}^{(0)}(f)=Q_{\rm Sob,p}^{(0)}(f).
\end{equation}

\subsection{The Riesz identity and calibration}

The next identity is needed only in the nonradial sectors.

\begin{lemma}
\label{lem:Riesz-identity-Rell}
Let \(\ell\ge1\), and set
\(h(r)=rU'(r)\).
Then
\begin{equation}
\label{eq:Riesz-identity-Rell}
Q_{\rm Sob,p}^{(\ell)}(h,g)
=
pR_\ell(g).
\end{equation}
\end{lemma}

\begin{proof}
Since
\[
-\Delta_p U_\lambda=\Lambda U_\lambda^{p^*-1},
\qquad  U_\lambda(x)=\lambda^{\frac{n-p}{p}}U(\lambda x).
\]
it follows that, for every radial test profile \(g\),
\begin{equation}
\label{eq:lambda-identify}
\int_0^\infty
|U_\lambda'|^{p-2}U_\lambda' g' r^{n-1}\,dr
=
\Lambda
\int_0^\infty
U_\lambda^{p^*-1}g r^{n-1}\,dr .
\end{equation}
Differentiating \eqref{eq:lambda-identify} at \(\lambda=1\), we obtain
\[
\begin{aligned}
&(p-1)
\int_0^\infty
|U'|^{p-2}\left(\frac{n-p}{p}U'+h'\right)g' r^{n-1}\,dr\\
&\qquad \qquad \qquad=
(p^*-1)\Lambda
\int_0^\infty
U^{p^*-2}\left(\frac{n-p}{p}U+h\right)g r^{n-1}\,dr .
\end{aligned}
\]
Combining this with \eqref{eq:lambda-identify} evaluated at \(\lambda=1\), we
find
\begin{equation}
\label{eq:h-dilation-identity}
\begin{aligned}
& (p-1)
\int_0^\infty
|U'|^{p-2}h'g' r^{n-1}\,dr
-
(p^*-1)\Lambda
\int_0^\infty
U^{p^*-2}hg r^{n-1}\,dr
\\
&\qquad \qquad \qquad
=p\int_0^\infty
|U'|^{p-2}U'g' r^{n-1}\,dr .
\end{aligned}
\end{equation}

We now insert \eqref{eq:h-dilation-identity} into the \(\ell\)-th sector
bilinear form. Using \eqref{eq:classical-p-sector},
\eqref{eq:Rell-local-p}, and \eqref{eq:h-dilation-identity}, we obtain
\[
\begin{aligned}
Q_{\rm Sob,p}^{(\ell)}(h,g)
&=
|\mathbb S^{n-1}|
\int_0^\infty
|U'|^{p-2}
\left[
(p-1)h'g'
+
\lambda_\ell\frac{hg}{r^2}
\right]
r^{n-1}\,dr\\
&\quad
-
|\mathbb S^{n-1}|
(p^*-1)\Lambda
\int_0^\infty U^{p^*-2}hg r^{n-1}\,dr \\
&=
p|\mathbb S^{n-1}|
\int_0^\infty
|U'|^{p-2}U'g' r^{n-1}\,dr
+
\lambda_\ell|\mathbb S^{n-1}|
\int_0^\infty
|U'|^{p-2}U'\frac{g}{r} r^{n-1}\,dr\\
&=
p|\mathbb S^{n-1}|
\int_0^\infty
|U'|^{p-2}U'
\left(g'+\frac{\lambda_\ell}{p}\frac{g}{r}\right)
r^{n-1}\,dr
=
pR_\ell(g).
\end{aligned}
\]
This is exactly \eqref{eq:Riesz-identity-Rell}.
\end{proof}

\begin{lemma}
\label{lem:Rell-norm-explicit}
Let
\[
I=\int_0^\infty |U'|^p r^{n-1}\,dr,
\qquad
h(r)=rU'(r).
\]
Then, for every \(\ell\ge1\),
\begin{equation}
\label{eq:Qhhell-explicit}
Q_{\rm Sob,p}^{(\ell)}(h,h)
=
|\mathbb S^{n-1}|
\bigl(\lambda_\ell-(n-p)\bigr)I .
\end{equation}
Moreover, for every \(\ell\ge2\),
\begin{equation}
\label{eq:Rell-dual-norm-explicit}
\|R_\ell\|^2_{(Q_{\rm Sob,p}^{(\ell)})^{-1}}
:=
\sup_{g\neq0}
\frac{R_\ell(g)^2}{Q_{\rm Sob,p}^{(\ell)}(g,g)}
=
\frac{|\mathbb S^{n-1}|}{p^2}
\bigl(\lambda_\ell-(n-p)\bigr)I .
\end{equation}
\end{lemma}

\begin{proof}
By \eqref{eq:Rell-local-p},
\[
\begin{aligned}
R_\ell(h)
&=
|\mathbb S^{n-1}|
\int_0^\infty
|U'|^{p-2}U'
\left(h'+\frac{\lambda_\ell}{p}\frac{h}{r}
\right)
r^{n-1}\,dr\\
&=
|\mathbb S^{n-1}|
\int_0^\infty
|U'|^{p-2}U'h' r^{n-1}\,dr
+
\frac{\lambda_\ell}{p}
|\mathbb S^{n-1}|I .
\end{aligned}
\]
Testing \eqref{eq:lambda-identify} with \(\lambda=1\) against \(U\), we obtain
\[
I=\int_0^\infty |U'|^p r^{n-1}\,dr
=
\Lambda
\int_0^\infty U^{p^*}r^{n-1}\,dr .
\]
Next, applying \eqref{eq:lambda-identify} with \(\lambda=1\) and \(g=h=rU'\),
we get
\[
\begin{aligned}
\int_0^\infty |U'|^{p-2}U'h' r^{n-1}\,dr
&=
\Lambda
\int_0^\infty U^{p^*-1}rU' r^{n-1}\,dr
=
\frac{\Lambda}{p^*}
\int_0^\infty r^n (U^{p^*})'\,dr\\
&=
-\frac{n}{p^*}\Lambda
\int_0^\infty U^{p^*}r^{n-1}\,dr
=
-\frac{n-p}{p}I .
\end{aligned}
\]
Therefore
\[
R_\ell(h)
=
\frac{|\mathbb S^{n-1}|}{p}
\bigl(\lambda_\ell-(n-p)\bigr)I .
\]
By Lemma~\ref{lem:Riesz-identity-Rell}, taking \(g=h\), we obtain
\[
Q_{\rm Sob,p}^{(\ell)}(h,h)
=
pR_\ell(h)
=
|\mathbb S^{n-1}|
\bigl(\lambda_\ell-(n-p)\bigr)I ,
\]
which proves \eqref{eq:Qhhell-explicit}.

It remains to identify the dual norm for \(\ell\ge2\). By Subsection~\ref{subsec:classical-nondegeneracy}, \(Q_{\rm Sob,p}^{(\ell)}\) has trivial
kernel in each sector \(\ell\ge2\). Hence its polarization defines a positive
definite inner product on the \(\ell\)-th sector. By
Lemma~\ref{lem:Riesz-identity-Rell}, together with the Cauchy inequality for
this inner product,
\[
R_\ell(g)^2
=
\frac1{p^2}
Q_{\rm Sob,p}^{(\ell)}(h,g)^2
\le
\frac1{p^2}
Q_{\rm Sob,p}^{(\ell)}(h,h)
Q_{\rm Sob,p}^{(\ell)}(g,g).
\]
Equality is attained for \(g=h\). Therefore
\[
\|R_\ell\|^2_{(Q_{\rm Sob,p}^{(\ell)})^{-1}}
=
\sup_{g\neq0}
\frac{R_\ell(g)^2}{Q_{\rm Sob,p}^{(\ell)}(g,g)}
=
\frac1{p^2}Q_{\rm Sob,p}^{(\ell)}(h,h)
=
\frac{|\mathbb S^{n-1}|}{p^2}
\bigl(\lambda_\ell-(n-p)\bigr)I .
\]
This proves \eqref{eq:Rell-dual-norm-explicit}.
\end{proof}

\begin{lemma}
\label{lem:degree-two-calibration-p}
We have
\begin{equation}
\label{eq:beta-two-explicit-p}
\beta_{2,p}
=
\frac{p^2}{|\mathbb S^{n-1}|(n+p)I}.
\end{equation}
\end{lemma}

\begin{proof}
Let \(B=B^T\) with \(\operatorname{tr}B=0\). Then
\[
Z_B(x):=x\cdot B\nabla U(x) =
rU'(r)\,\theta\cdot B\theta .
\]
Since \(\operatorname{tr}B=0\), the function \(\theta\mapsto \theta\cdot B\theta\)
is a spherical harmonic of degree two on \(\mathbb S^{n-1}\).

Moreover, since
\(\det(e^{tB})=e^{t\operatorname{tr}B}=1\), the curve \(U_t(x)=U(e^{tB}x)\)
is generated by volume-preserving affine transformations. Hence both the
affine energy and the \(L^{p^*}(\mathbb R^n)\)-norm are invariant along this
curve. Differentiating twice at \(t=0\), and using
\[
\frac{d}{dt}\bigg|_{t=0}U(e^{tB}x)
=
x\cdot B\nabla U(x)
=
Z_B(x),
\]
we obtain
\begin{equation}
\label{eq:Qzb=0}
 Q_{\rm aff,p}(Z_B)=0.
\end{equation}

Choose \(B\neq0\) with \(B=B^T\), \(\operatorname{tr}B=0\), and write
\[
\theta\cdot B\theta
=
c_B Y_{2,m}(\theta),
\qquad
\left\langle Y_{2,m}^2\right\rangle_{\mathbb S^{n-1}}=1,
\qquad c_B\neq0.
\]
Then
\[
Z_B(x)=c_B h(r)Y_{2,m}(\theta),
\qquad
h(r)=rU'(r).
\]
By \eqref{eq:Q-aff-full-sector-decomposition}, \eqref{eq:Qzb=0}, and
Lemma~\ref{lem:Riesz-identity-Rell},
\[
0=  Q_{\rm aff,p}^{(2)}(h)
=
Q_{\rm Sob,p}^{(2)}(h,h)
-
\beta_{2,p}R_2(h)^2=  Q_{\rm Sob,p}^{(2)}(h,h)
-
\beta_{2,p}
\frac1{p^2}
Q_{\rm Sob,p}^{(2)}(h,h)^2 .
\]
Combining this with Lemma~\ref{lem:Rell-norm-explicit}, we obtain
\[
\beta_{2,p}
=
\frac{p^2}{Q_{\rm Sob,p}^{(2)}(h,h)}
        =
\frac{p^2}{|\mathbb S^{n-1}|(n+p)I},
\]
which is exactly \eqref{eq:beta-two-explicit-p}. The proof is complete.
\end{proof}

\subsection{The degree-two affine kernel}

We now identify the kernel of the affine Hessian in the degree-two sector. For
convenience, we write
\[
Q_{\rm aff,p}[\phi_\ell]
:=
\sum_{m=1}^{d_\ell}
Q_{\rm aff,p}^{(\ell)}(f_{\ell,m}).
\]

\begin{lemma}
\label{lem:degree-two-affine-sector-p}
In the physical degree-two sector,
\[
\left\{
\phi_2:
Q_{\rm aff,p}[\phi_2]=0
\right\}
=
\left\{
x\cdot B\nabla U:
B=B^T,\ \operatorname{tr}B=0
\right\}.
\]
\end{lemma}

\begin{proof}
Set \(h(r)=rU'(r)\).
Fix one degree-two spherical harmonic \(Y_{2,m}\). By
\eqref{eq:Q-aff-full-sector-decomposition},
\begin{equation}
\label{eq:degree-2}
Q_{\rm aff,p}^{(2)}(f)
=
Q_{\rm Sob,p}^{(2)}(f)
-
\beta_{2,p}R_2(f)^2 .
\end{equation}
Since the classical \(p\)-Sobolev Hessian
\eqref{eq:classical-p-homogeneous-Hessian} has no kernel in the degree-two
sector, \(Q_{\rm Sob,p}^{(2)}\) is positive definite. We write
\[
\|f\|_2^2:=Q_{\rm Sob,p}^{(2)}(f),
\qquad
\langle f,g\rangle_2:=Q_{\rm Sob,p}^{(2)}(f,g).
\]
Hence, by Lemma~\ref{lem:Riesz-identity-Rell},
Lemma~\ref{lem:Rell-norm-explicit},
Lemma~\ref{lem:degree-two-calibration-p}, and \eqref{eq:degree-2},
\begin{equation}
\label{eq:nonnegative-2-sector}
Q_{\rm aff,p}^{(2)}(f)
=
Q_{\rm Sob,p}^{(2)}(f)
-
\beta_{2,p}R_2(f)^2
=
\|f\|_2^2
-
\frac{|\langle f,h\rangle_2|^2}{\|h\|_2^2}\ge0,
\end{equation}
with equality if and only if \(f\in\operatorname{span}\{h\}\).

Now let
\[
\phi_2(r,\theta)
=
\sum_{m=1}^{d_2}f_m(r)Y_{2,m}(\theta).
\]
Then \eqref{eq:nonnegative-2-sector} yields
\[
Q_{\rm aff,p}[\phi_2]
=
\sum_{m=1}^{d_2}Q_{\rm aff,p}^{(2)}(f_m)=0 \quad\Longleftrightarrow\quad
f_m\in\operatorname{span}\{rU'\}
\qquad
\text{for every }m.
\]
Every degree-two spherical harmonic is the restriction of a trace-free
quadratic polynomial. Hence there exists \(B=B^T\), \(\operatorname{tr}B=0\),
such that
\[
\phi_2(r,\theta)
=
rU'(r)\sum_{m=1}^{d_2}c_mY_{2,m}(\theta)
=
rU'(r)\,\theta\cdot B\theta
=
x\cdot B\nabla U(x).
\]
The converse follows from \eqref{eq:Qzb=0}. This proves the lemma.
\end{proof}

\subsection{The higher even sectors}

\begin{lemma}
\label{lem:high-even-sector-p}
For every even \(\ell\ge4\), there exists \(\eta_{n,p}>0\), independent of
\(\ell\), such that,
\[
Q_{\rm aff,p}[\phi_\ell]
\ge
\eta_{n,p}Q_{\rm Sob,p}[\phi_\ell],
\qquad
\ell=4,6,8,\dots .
\]
\end{lemma}

\begin{proof}
Let \(\ell=2m\), with \(m\ge2\). We first work in a single angular component.
By Lemma~\ref{lem:Rell-norm-explicit},
\[
\|R_{2m}\|^2_{(Q_{\rm Sob,p}^{(2m)})^{-1}}
=
\frac{|\mathbb S^{n-1}|}{p^2}
\bigl(\lambda_{2m}-(n-p)\bigr)I .
\]
Therefore, by \eqref{eq:beta-ratio-p} and
Lemma~\ref{lem:degree-two-calibration-p},
\[
\begin{aligned}
\beta_{2m,p}
\|R_{2m}\|^2_{(Q_{\rm Sob,p}^{(2m)})^{-1}}
&=
\frac{\beta_{2m,p}}{\beta_{2,p}}
\frac{\lambda_{2m}-(n-p)}{n+p}\\
&=
\frac{\lambda_{2m}-(n-p)}{n+p}
\prod_{j=1}^{m-1}
\left(
\frac{\left|j-\frac p2\right|}
{j+\frac n2+\frac p2}
\right)^2
=:\nu_m .
\end{aligned}
\]
A direct computation gives, for \(m\ge2\),
\begin{equation}
\label{eq:nu-de}
\frac{\nu_{m+1}}{\nu_m}
=
\frac{\lambda_{2m+2}-(n-p)}
{\lambda_{2m}-(n-p)}
\left(
\frac{|m-\frac p2|}{m+\frac n2+\frac p2}
\right)^2
<1.
\end{equation}
Since \(\lambda_4=4(n+2)\),
\[
\nu_2
=
\frac{3n+p+8}{n+p}
\left(
\frac{p-2}{n+p+2}
\right)^2<1.
\]
Hence, for every even \(\ell=2m\ge4\), \eqref{eq:nu-de} implies
\[
\beta_{\ell,p}
\|R_\ell\|^2_{(Q_{\rm Sob,p}^{(\ell)})^{-1}}
\le
1-\eta_{n,p},
\qquad  \eta_{n,p}:=1-\nu_2.
\]
By the definition of the dual norm \eqref{eq:Rell-dual-norm-explicit},
\[
\beta_{\ell,p}R_\ell(f)^2
\le
\beta_{\ell,p}  \|R_\ell\|^2_{(Q_{\rm Sob,p}^{(\ell)})^{-1}}
 Q_{\rm Sob,p}^{(\ell)}(f)
\le
(1-\eta_{n,p})Q_{\rm Sob,p}^{(\ell)}(f).
\]
Therefore, by \eqref{eq:Q-aff-full-sector-decomposition},
\[
Q_{\rm aff,p}^{(\ell)}(f)
=
Q_{\rm Sob,p}^{(\ell)}(f) -
\beta_{\ell,p}R_\ell(f)^2
\ge
\eta_{n,p}Q_{\rm Sob,p}^{(\ell)}(f).
\]
Summing this one-component estimate over all components yields
\[
Q_{\rm aff,p}[\phi_\ell]
\ge
\eta_{n,p}Q_{\rm Sob,p}[\phi_\ell].
\]
This completes the proof.
\end{proof}

\subsection{Kernel identification}

By \eqref{eq:redial-kernel} and \eqref{eq:radial-affine-correction-vanishes},
\[
Q_{\rm aff,p}[\phi_0]
=
Q_{\rm Sob,p}[\phi_0]=0
\quad\Longleftrightarrow\quad
\phi_0\in\operatorname{span}\{U,Z_0\},
\]

In the first angular sector, by \eqref{eq:rho-odd-zero-p}, \( d_{1,p}=0\), so
by \eqref{eq:1st-kernel}, \eqref{eq:Q-aff-full-sector-decomposition},
\[
Q_{\rm aff,p}[\phi_1]
=
Q_{\rm Sob,p}[\phi_1]=0
\quad\Longleftrightarrow\quad
\phi_1\in
\operatorname{span}
\left\{
\partial_{x_1}U,\dots,\partial_{x_n}U
\right\}.
\]

In the degree-two sector, Lemma~\ref{lem:degree-two-affine-sector-p} gives
\[
Q_{\rm aff,p}[\phi_2]=0
\quad\Longleftrightarrow\quad
\phi_2\in
\left\{
x\cdot B\nabla U:
 B=B^T,\ \operatorname{tr}B=0
\right\}.
\]

For odd sectors \(\ell\ge3\), \eqref{eq:rho-odd-zero-p} gives \( d_{\ell,p}=0\),
and hence
\[
Q_{\rm aff,p}[\phi_\ell]
=
Q_{\rm Sob,p}[\phi_\ell],
\qquad
\ell=3,5,7,\dots .
\]
Since the classical Hessian has no kernel in sectors \(\ell\ge2\), these odd
sectors are strictly positive.

For even sectors \(\ell\ge4\), Lemma~\ref{lem:high-even-sector-p} yields
\[
Q_{\rm aff,p}[\phi_\ell]
\ge
\eta_{n,p}Q_{\rm Sob,p}[\phi_\ell],
\qquad
\ell=4,6,8,\dots .
\]
Thus no additional zero modes arise in the higher even sectors.

Consequently,
\begin{equation}
\label{eq:p-affine-kernel}
\ker Q_{\rm aff,p}
=
\operatorname{span}
\left\{
U,\ Z_0,\ \partial_{x_1}U,\dots,\partial_{x_n}U,\ x\cdot B\nabla U
\right\},
\end{equation}
where \( B=B^T\), \(  \operatorname{tr}B=0\).

\noindent\textbf{Conflict of interest:}
Authors state no conflict of interest.

\noindent\textbf{Data Availability Statement:}
Data sharing is not applicable to this article as no datasets were generated or analysed during the current study.

\noindent{\bf Acknowledgement.}
G-D.~Li was supported by NSFC (No.12561019).

\providecommand{\href}[2]{#2}
\providecommand{\arxiv}[1]{\href{http://arxiv.org/abs/#1}{arXiv:#1}}
\providecommand{\url}[1]{\texttt{#1}}
\providecommand{\urlprefix}{DOI }


\end{document}